\documentclass[11pt]{amsart}

\usepackage{xspace}

\usepackage{amsmath}
\usepackage{amssymb}
\usepackage{amsthm}

 \usepackage{xcolor}

 \usepackage{url}
 \usepackage{graphicx} 
 \usepackage[all]{xy} 

\allowdisplaybreaks 

\newtheorem{teo}{teorema}[section]
\newtheorem{theorem}[teo]{Theorem}

\newtheorem{corollary}[teo]{Corollary}
\newtheorem{lema}[teo]{Lemma}
\newtheorem{lemma}[teo]{Lemma}

\newtheorem{proposition}[teo]{Proposition}

\theoremstyle{definition}
\newtheorem{definition}[teo]{Definition}
\newtheorem{example}[teo]{Example}

\theoremstyle{remark}
\newtheorem{remark}{Remark}

\numberwithin{figure}{section}

\newcommand{\PP}{\mathrm{P}}
\newcommand{\scat}{\mathop{\mathrm{scat}}}

\newcommand{\TC}{\mathop{\mathrm{TC}}}

\newcommand{\ctg}{\sim_c}   
\newcommand{\ctgcl}{\sim}   

\newcommand{\inv}{^{-1}}	
\newcommand{\simc}{\sim_c}	

\newcommand{\Sg}{\mathrm{Sg}}	
\newcommand{\hSg}{\mathrm{hSg}}
\newcommand{\ev}{\mathrm{ev}}
\newcommand{\pr}{\mathrm{pr}}

\newcommand{\eqdef}{\mathrel{\mathop:}=}
\newcommand{\cat}{\mathrm{cat}}

\title[Simplicial fibrations]{Simplicial fibrations}
\thanks{The first and the fourh authors were partially supported by
MINECO Spain Research Project MTM2015--65397--P and Junta de
Andaluc\'ia Research Groups FQM--326 and FQM--189. The second and
third authors were partially supported by MINECO-FEDER research
project MTM2016--78647--P}

\author[D.~Fern\'andez-Ternero et al.]{D.~Fern\'andez-Ternero}
\address[D. Fern\'{a}ndez-Ternero, J.A. Vilches]{\newline
\indent Departamento de Geometr\'ia y Topolog\'ia, Universidad de Sevilla, Spain}
\email{desamfer@us.es}
\email{vilches@us.es}

\author[]{J.M. Garc\'{\i}a Calcines}
\address[J.M. Garc\'{\i}a Calcines]{\newline
\indent Dpto. de Matem\'aticas, Estad\'{\i}stica e Investigaci\'on Operativa. Universidad de La Laguna, Spain}
\email{jmgarcal@ull.es}

\author[]{E.~Mac\'{\i}as-Virg\'os}
\address[E. Mac\'{i}as-Virg\'os] {\newline \indent Institute of Mathematics, University of Santiago de Compostela, Spain.}
\email{quique.macias@usc.es}

\author[]{J.A.~Vilches}

\date{\today}
\begin{document}
\maketitle

\begin{abstract}
We undertake a systematic study of the notion of fibration in the
setting of abstract simplicial complexes, where the concept of
``homotopy'' has been replaced by that of ``contiguity''. Then a
fibration will be a simplicial map satisfying the ``contiguity
lifting property''. This definition turns out to be equivalent to
a known notion introduced by G. Minian, established in terms of a
cylinder construction $K \times I_m$. This allows us to prove
several properties of simplicial fibrations which are analogous to
the classical ones in the topological setting, for instance: all
the fibers of a fibration have the same strong homotopy type, a
notion that has been recently introduced by Barmak and Minian; any
fibration with a strongly collapsible base is fibrewise trivial;
and some other ones. We introduce the concept of ``simplicial
finite-fibration'', that is, a map which has the contiguity
lifting property only for finite complexes. Then, we prove that
the path fibration $PK \to K\times K$ is a finite-fibration, where
PK is the space of Moore paths introduced by M. Grandis. This
important result allows us to prove that any simplicial map
factors through a finite-fibration, up to a P-homotopy
equivalence. Moreover, we introduce a definition of ``\v{S}varc
genus'' of a simplicial map and, and using the properties stated
before, we are able to compare the \v{S}varc genus of path
fibrations with the notions of simplicial LS-category and
simplicial topological complexity introduced by the authors in
several previous papers. Finally, another key result is a
simplicial version of a Varadarajan result for fibrations,
relating the LS-category of the total space, the base and the
generic fiber.
\end{abstract}

\setcounter{tocdepth}{1}

\tableofcontents


 \section{Introduction}

In recent years there has been a renovated interest in abstract
simplicial complexes, as a setting which is well suited for
discretizing topological invariants and for designing computer
algorithms. Under this point of view and boosted by the
increasing computer capacities, several classical theories have
been developed, thus providing new powerful tools
like persistent homology \cite{PERSISTENT} or discrete Morse
theory \cite{DISCRETEMORSE}, which are being applied in robotics,
neural networks or big data mining.

However, there is a lack of development of other ideas in
this new field of ``applied algebraic topology'', like
Lusternik-Schnirelmann category or topological complexity, which
classically needed the use of notions such as homotopy, fibrations
or cofibrations.

In the framework of abstract simplicial complexes, the classical
notion of ``contiguity'' between simplicial maps \cite{S}  plays
the role that ``homotopy'' plays in the context of topological
spaces. This notion has received new attention in the last years
after the work of J. Barmak and G. Minian \cite{B-M}. They showed
that the equivalence under contiguity classes is the same as the
equivalence by ``strong collapses'', a highly interesting idea
which is related on one hand with the classical Whitehead
collapses, and on the other hand with the theory of posets and
finite topological spaces \cite{LIBROBARMAK}.

Using the ideas above, several of the authors have recently
introduced a notion of LS-category in the simplicial setting,
which generalizes the well known notion of ``arboricity'' in graph
theory \cite{F-M-V, F-M-M-V1}. Moreover,  we also introduced a
notion of topological complexity,  defined in purely combinatorial
terms \cite{F-M-M-V2}. Both invariants have similar properties to
the classical ones and also new results arise.

As a collateral result, cofibrations were studied in
\cite{F-M-M-V1}, but a systematic study of the notion of
simplicial fibration was lacking. This study is the aim of the
present paper.

The contents are as follows:

In Section \ref{S2} we recall the basic notions of simplicial
complexes  and contiguity classes.

In Section \ref{S3} we introduce two possible definitions of {\em
simplicial fibration} in terms of a {\em contiguity lifting
property} and we show that in fact they are equivalent to a third
one introduced by Minian in \cite{M2005}: a simplicial map
$p\colon E \to B$ is a simplicial fibration if for any simplicial
map $H\colon K \times I_m \to B$ and any simplicial map $f:K\times
\{0\} \to E$, there is a simplicial map $\tilde H\colon K \times
I_m \to E$ such that $p\circ \tilde H=H$ and $\tilde H \circ
i_0^m= f$ (see Definition \ref{DEF3}). There is a more general
notion of {\em simplicial finite-fibration} if we limit the
lifting property to {\em finite} complexes $K$ in the definition
above. It will be necessary to obtain several key results along the paper.


In Section \ref{S4} we give several basic examples and
constructions, including products and pullbacks of simplicial
fibrations. Then, we introduce (Section \ref{S5}) the notion of
{\em Moore path} and the space $\PP K$ of Moore paths on a
simplicial complex $K$. This notion has been developed in \cite{G}
by M. Grandis. The main result of this section is that the path map $\PP K \to
K\times K$ is a simplicial finite-fibration (Theorem
\ref{MAINPATH}).

In Section \ref{GENFIBER} we prove the important result that all
the fibers of a simplicial fibration have the same strong homotopy
type. In the same line, we adapt another classical result by
showing (Section \ref{S7}) that a simplicial fibration with a
collapsible base is trivial. Here, ``collapsible'' means that
there is a finite sequence of strong collapses and expansions
transforming the base onto a point. Our next result (Section
\ref{S8}) is a simplicial version of Varadarajan's theorem (see
Theorem \ref{Varadarajan}) relating the LS-category of the total
space, the base and the generic fiber of a fibration.

In the last sections of the paper we introduce several new ideas.
The first one is based on a notion of ``P-homotopy'' modelled on
Moore paths (see Definition \ref{P-homotop}), which allows us to
prove the $P$-equivalence of the complexes $K$ and $\PP K$ and to
give a general result  about the factorization of any map into a
P-equivalence and a finite-fibration (Section \ref{FACTORIZ}).

On the other hand, the $P$-homotopy equals
the usual contiguity property for finite complexes. This allows us to define in Section
\ref{SVARCGENUS} a general notion of \v{S}varc genus for
simplicial maps and to discuss its relationship with the
simplicial LS-category and the discrete topological complexity
introduced in our previous papers \cite{F-M-V, F-M-M-V1,F-M-M-V2}.


\section{Simplicial complexes and contiguity}\label{S2}

We start by briefly recalling the notions of simplicial complex
and contiguity. We are assuming that the reader
is familiarized with these notions as well as others that will be
appearing throughout the paper (see, for instance, \cite{K,S} for
more details on this topic).

\begin{definition}
An {\em (abstract) simplicial complex} is a set $V$ together with
a collection $K$ of finite subsets of $V$ such that if $\sigma\in
K$ and $\tau\subseteq \sigma$ then $\tau\in K$.
\end{definition}
Notice that $K$ is not necessarily finite in the above definition.
As usual, $K$ will denote the simplicial complex and $V(K)$ the
corresponding vertex set.

\begin{definition}
Given two simplicial complexes $K$ and $L$, a {\em simplicial map}
from $K$ to $L$ is a set map $\varphi \colon V(K)\to V(L)$ such
that if $\sigma\in K$ then $\varphi(\sigma)\in L$.
\end{definition}

\begin{definition}
Let $K, L$ be two simplicial complexes.  Two simplicial maps
$\varphi,\psi\colon K\to L$ are {\em contiguous} \cite[p.~130]{S}
if, for any simplex $\sigma\in K$, the set
$\varphi(\sigma)\cup\psi(\sigma)$ is a simplex of $L${;} that is,
if $v_0,\dots,v_k$ are the vertices of $\sigma$ then the vertices
$f(v_0),\dots,f(v_k),g(v_0),\dots,g(v_k)$ span a simplex of $L$.
\end{definition}

This relation, denoted by $\varphi \ctg \psi$, is reflexive and
symmetric, but in general it is not transitive. In order to
overcome this fact we use the notion of contiguity class.

\begin{definition}\label{MSTEPS}
Two simplicial maps $\varphi,\psi\colon K \to L$ are in the same
{\em contiguity class} with $m$ steps, denoted by $\varphi\ctgcl
\psi$, if there is a finite sequence {$$\varphi=\varphi_0 \ctg
\cdots \ctg \varphi_m=\psi$$} of contiguous {simplicial}  maps
$\varphi_i\colon K \to L$, $0\leq i\leq m$.
\end{definition}

It is straightforward to prove the following:

\begin{lema}\label{lemacontig}
Let $f \colon K \to L$ be a simplicial map between the simplicial
complexes $K$ and $L$. If $g\colon V(K) \to V(L)$ is a map such
that $f\simc g$ then $g$ is a simplicial map.
\end{lema}


Now we recall a formal notion of combinatorial homotopy
introduced by Minian in \cite{M2005}. First of all we need a
triangulated version of the real interval $[0,n]$.

\begin{definition}
For $n\geq 1$, let $I_n$ be the one-dimensional simplicial complex
whose vertices are the integers $\{0,\dots,n\}$ and the edges are
the pairs $\{j,j+1\}$, for $0\leq j<n$.
\end{definition}

\begin{definition}\cite[Definition 4.25]{K}
Let $K$ and $L$ be two simplicial complexes. The {\em categorical
product} $K\times L$ is the simplicial complex whose set of
vertices is $V(K\times L)=V(K)\times V(L),$ and  whose simplices
are given by the rule: $\sigma\in K\times L$ if and only if
$\pr_1(\sigma)\in K$ and $\pr_2(\sigma)\in L$, where $\pr_1,\pr_2$
are the canonical projections.
\end{definition}

Given two simplicial maps, Minian proved that belonging to the
same contiguity class is equivalent to the existence of a
simplicial homotopy between them, modelled by a simplicial
cylinder.

\begin{proposition}\cite[Prop.~2.16]{M2005}\label{MINIANCONT}
Two simplicial maps $f,g\colon K \to L$ are in the same contiguity
class, with $m$ steps, $f\sim g$, if and only if there exists some $m\geq 1$ and
some simplicial map  $H\colon K \times I_m\to L$ such that
$H(v,0)=f(v)$ and $H(v,m)=g(v)$, for all vertices $v\in K$.
\end{proposition}

\begin{remark}
The preceding proposition holds when we consider $K\times I_m$
to be the {\em categorical} product, but notice that the proof
does not work for the more usual notion of simplicial product,
namely the so-called simplicial {\em cartesian} product (see
\cite{K}).
\end{remark}

Barmak and Minian introduced \cite{LIBROBARMAK, B-M} the so-called
strong homotopy type for simplicial complexes. Two simplicial
complexes have the same strong homotopy type, denoted by $K\sim
L$, if they are related by a finite sequence of two kind of
simplicial moves, namely, strong collapses and expansions. An
elementary strong collapse consists of removing the open star
around a dominated vertex, where a vertex $v$ is dominated by
another vertex $v^\prime$ if every maximal simplex that contains
$v$ also contains $v^\prime$. A complex is called strongly
collapsible if it has the same strong homotopy type of a point.

Strong homotopy type is deeply related to the notion of contiguity
between simplicial maps. More precisely, the following result
holds:

\begin{proposition}\cite[Cor.~2.12]{B-M}
Two simplicial complexes $K$ and $L$ have the same strong homotopy
type if and only if there are simplicial maps $\varphi\colon K\to
L$ and $\psi\colon L \to K$ such that $\psi\circ\varphi\sim 1_K$
and $\varphi\circ\psi\sim 1_L$.
\end{proposition}


\section{Definitions of simplicial fibration}\label{S3}
The goal of this section is to establish a notion of fibration
in the simplicial context. As we shall see, there are several
options of doing this, depending on the particular kind of lifting
property we deal with.

Our first definition of fibration in the simplicial context
corresponds to simplicial maps with the {\em contiguity lifting
property} with respect to any simplicial complex.
\begin{definition}\label{DEF1}
A simplicial map $p\colon E \to B$ is a {\em type I simplicial
fibration} if  for any simplicial complex $K$ (finite or not),
given  any two  contiguous simplicial maps $f,g \colon K\to B$, $f
\simc g$, and  any map $\tilde f \colon K \to E$ such that $p\circ
\tilde f= f$, there exists a simplicial map $\tilde g \colon K \to
E$ such that $\tilde f$ and $\tilde g$ are contiguous, $\tilde
f\simc \tilde g$, and $p\circ \tilde g=g$ (see Figure
\ref{LIFTCONT}).

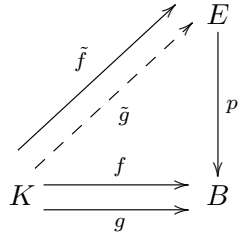
\begin{figure}[h]\label{LIFTCONT}
\begin{center}
\begin{minipage}{10cm}
\xymatrix{
&&\ E\ \ar[dd]^p \\
&&\\
K\ar@<+1ex>[rr]^{\ f\ }\ar@<-1ex>[rr]_{\ g\
}\ar@<+3ex>[rruu]^{\
\tilde f}\ar@{-->}@<+1ex>[rruu]_{\tilde g\ }&&\ B\ }
\end{minipage}
\end{center}
\caption{Contiguity lifting  property}
\end{figure}
\end{definition}

A second option to give a simplicial notion of fibration consists
in generalizing the first definition to contiguity classes with a
given number of steps (see Definition \ref{MSTEPS}).

\begin{definition}\label{DEF2}
A simplicial map $p\colon E \to B$ is a {\em type II simplicial
fibration} if for any simplicial complex $K$, for any two
simplicial maps $f,g \colon K\to B$ in the same contiguity class,
$f \sim g$, with $m$ steps, and for any map $\tilde f \colon K \to
E$ such that $p\circ \tilde f= f$, there exists a simplicial map
$\tilde g \colon K \to E$ such that $\tilde f$ and $\tilde g$ are
in the same contiguity class with $m$ steps, $\tilde f\sim \tilde
g$, and $p\circ \tilde g=g$.
\end{definition}

At this point we are ready for the third definition of simplicial
fibration in terms of the notion of homotopy, introduced by Minian
(see Prop. \ref{MINIANCONT}).

\begin{definition}\label{DEF3}
The map $p \colon E \to B$ is a {\em type III simplicial
fibration} if given simplicial maps $H\colon K\times I_m\to B$ and
$\varphi\colon K\times \{0\} \to E$ as in the following
commutative diagram:
\begin{equation}\label{Type3}
\xymatrix{
K\times \{ 0\}\ \ar@{^{(}->}[d]_{i_0^m}\ar[r]^{\varphi}&\ E\ \ar[d]^p \\
K\times I_m\ \ar@{-->}@<+0ex>[ru]^{\widetilde{H}\ }\ar[r]^{\ H\ }&\ B\ }
\end{equation}

\noindent there exists a simplicial map $\widetilde{H} \colon K
\times I_m \to E$ such that $\widetilde{H}\circ i_0^m=\varphi$ and
$p\circ \widetilde{H}=H$.
\end{definition}

\begin{theorem}\label{equiv}
The three definitions of simplicial fibration are equivalent.
\end{theorem}

\begin{proof}$\mbox{ }$\par
\medskip
\textbf{Type I $\Rightarrow$   Type II:} Assuming that the
simplicial map $p \colon E \to B$ satisfies Definition~\ref{DEF1},
let us consider a simplicial complex $K$ and two simplicial maps
$f,g \colon K \to B$ in the same contiguity class with $m$ steps,
that is, there exists a finite sequence of direct contiguities
$$f =\varphi_0 \simc \varphi_1 \simc \cdots \simc \varphi_m = g.$$
Taking a simplicial map $\widetilde{f}$ such that $p\circ \tilde
f= f$, by Definition~\ref{DEF1}, there is a simplicial map
$\widetilde{\varphi_1} \colon K \to E$ such that $\widetilde{f}
\simc \widetilde{\varphi_1}$ and $p\circ\widetilde{\varphi_1} =
\varphi_1$. Iterating the same argument for every $\varphi_i$ with
$i=1,\dots, m$, we obtain a finite sequence of direct
contiguities, $\widetilde f =\widetilde \varphi_0 \simc
\widetilde\varphi_1 \simc \cdots \simc \widetilde \varphi_m$,
where $p\circ\widetilde{\varphi_i} = \varphi_i $, $i=1,\dots,m$.
In particular, for $i=m$ we get a simplicial map $\widetilde
g=\widetilde \varphi_m$ satisfying $\tilde f\sim \tilde g$ with
$m$ steps and $p\circ \tilde g=g$ since $\varphi_m = g$. So, we conclude that $P$ is a type II simplicial fibration.

\medskip

\textbf{Type II $\Rightarrow$   Type III:} Let us assume that the
simplicial map $p \colon E \to B$ satisfies Definition
~\ref{DEF2}. Consider a simplicial complex $K$ and simplicial maps
$\varphi \colon K\times \{0\} \to E$ and $H \colon K\times I_m \to
B$ such that Diagram \eqref{Type3} is commutative.

Now, we define $\varphi_i \colon K \to B$ by
$\varphi_i(v)=H(v,i)$, where $v\in V(K)$ and $i\in V(I_m)$. By
means of Lemma \ref{lemacontig}, we only need to prove the
contiguity condition. These maps are contiguous because given
$\tau \in K$ the following fact holds true:
\begin{equation}\label{TRUENUEVO}
\varphi_i(\tau)\cup \varphi_{i+1}(\tau)=H(\tau \times \{i,i+1\})\in B
\end{equation}
since $H$ is a simplicial map.
Then, by hypothesis, there exists a finite chain of  simplicial maps $\widetilde \varphi_i\colon K \to E$ and
direct contiguities, with $m$ steps,
$\widetilde
\varphi_0 \simc \widetilde\varphi_1 \simc \cdots \simc \widetilde
\varphi_m$,  such that $p\circ \widetilde\varphi_i=\varphi_i$, for all $i=0,\dots,m$.

Hence, the map $\widetilde{H} \colon K\times I_m \to E$ given by
$\widetilde{H}(v, i)=\widetilde{\varphi}_i(v)$ is simplicial, by an argument analogous to \eqref{TRUENUEVO}, and satisfies
$\widetilde{H}\circ i_0^m= \varphi$ and $p\circ \widetilde{H}=H$. So, $H$ satisfies Definition \ref{DEF3}.
%

\medskip

\textbf{Type III $\Rightarrow$  Type I:} Let us assume that the
simplicial map $p \colon E \to B$ satisfies Definition~\ref{DEF3}.
Consider a simplicial complex $K$ and two contiguous simplicial
maps $f,g \colon K \to B$. Now, by Proposition~\ref{MINIANCONT},
with $m=1$, there exist a homotopy $H \colon K\times I_1 \to B $
such that $H(v,0)=f(v)$ and $H(v,1)=g(v)$ for all $v\in V(K)$.

Consider a simplicial map $\tilde f \colon K \to E$ such that
$p\circ \tilde f= f$. By Definition~\ref{DEF3}, there is a
simplicial map $\widetilde{H} \colon K \times I_1 \to E$ such that
$\widetilde{H}\circ i_0^1=\varphi$, where $p\circ\varphi (v,0)=
\widetilde{f}(v)$ for all $v\in V(K)$, and $p\circ
\widetilde{H}=H$. Let $\widetilde{g} \colon K \to E$ given by
$\widetilde{g}(v)=\widetilde{H}(v,1)$, where $v\in V(K)$. By
Proposition~\ref{MINIANCONT}, we conclude that
$\widetilde{f}\simc\widetilde{g}$.
\end{proof}

\begin{remark}
Notice that the complex $K$ that we considered in the definitions
above may not be finite.
\end{remark}

\begin{remark}\label{FINITELIFT} Observe that it is possible to
restrict these definitions to the cases where $K$ is finite. This
allows us to introduce the corresponding notions of
\emph{simplicial finite-fibration} of type I, II and III, which
are equivalent by the finite version of Theorem \ref{equiv}.
\end{remark}


\section{Examples and properties}\label{S4}

In this section we will introduce some important examples of
simplicial fibrations.  The following proposition will give us the
first basic ones. Notice that, unless otherwise specified, we will
use the notion of type III simplicial fibration given in
Definition \ref{DEF3}.

\begin{proposition}\label{firstprop} ${}$

\begin{enumerate}
\item[(i)] Any simplicial isomorphism is a simplicial fibration.

\item[(ii)] If $*$ denotes the one-vertex simplicial complex,
then the constant simplicial map $E\rightarrow *$ is a simplicial fibration,
for any  simplicial complex $E$.

\item[(iii)] The composition of simplicial fibrations is a simplicial
fibration.
\end{enumerate}
\end{proposition}

\begin{proof}
(i) is easily checked; indeed, if $p\colon E\stackrel{\cong
}{\longrightarrow }B$ is a simplicial isomorphism, then, given any
$m\geq 0$ and a commutative diagram $p\circ \varphi =H\circ
i^m_0,$ we have that the composition $\widetilde{H}\eqdef
p^{-1}\circ H$ satisfies the required conditions. Item (ii) is
also straightforward, since for any simplicial map $\varphi \colon
K\times \{0\} \to E$ the vertex map $K\times I_m\rightarrow E$,
given by $(v,i)\mapsto \varphi (v,0)$, is simplicial.

Now, in order to prove (iii), consider $p\colon E\rightarrow B$
and $q\colon B\rightarrow C$ simplicial fibrations, $m\geq 1$, and
the following commutative diagram of simplicial maps:
$$
\xymatrix{ {K\times \{0\}} \ar[rr]^{\varphi } \ar@{^{(}->}[d]_{i_0^m} & & {E}
\ar[d]^{q\circ p} \\ {K\times I_m} \ar[rr]_G & & {C} }
$$
Since $q$ is a simplicial fibration we can consider a simplicial
map $\widehat{G}\colon K\times I_m\rightarrow B$ satisfying
$\widehat{G}\circ i_0^m=p\circ \varphi $ and $q\circ
\widehat{G}=G$. Finally, as $p$ is a simplicial fibration we can
also consider a simplicial map $\widetilde{G}\colon K\times
I_m\rightarrow E$ such that $\widetilde{G}\circ i_0^m=\varphi $
and $p\circ \widetilde{G}=\widehat{G}$. From these conditions we
have that $(q\circ p)\circ \widetilde{G}=q\circ \widehat{G}=G$.
\end{proof}

For the next result we need to recall the pullback construction
for simplicial complexes. Given any pair of simplicial maps
$f\colon K\rightarrow M$ and $g\colon L\rightarrow M$ their
pullback is given as the following diagram:
$$\xymatrix{
{K\times _ML} \ar[rr]^{f'} \ar[d]_{g'} & & {L} \ar[d]^g \\ {K}
\ar[rr]_f & & {M} }$$ \noindent where $K\times _M L$ is the full
simplicial subcomplex of $K\times L$ whose underlying vertex set
is given by those pairs of vertices $(v,w)\in K\times L$
satisfying $f(v)=g(w)$. The induced simplicial maps $f'$ and $g'$
are given by $f'(v,w)=w$ and $g'(v,w)=v$. It is plain to check
that this construction is the pullback of $f$ and $g$ in the
category of simplicial complexes.

\begin{proposition}\label{pullback}
Let $p\colon E\rightarrow B$ be a simplicial fibration and
$f\colon K\rightarrow B$ any simplicial map. Then the simplicial
map $p^\prime \colon K\times _B E\rightarrow K$ induced by $p$ in
the pullback
$$\xymatrix{
{K\times _BE} \ar[rr]^{f'} \ar[d]_{p'} & & {E} \ar[d]^p \\ {K}
\ar[rr]_f & & {B} }$$ \noindent is also a simplicial fibration.
\end{proposition}

\begin{proof}
Take any commutative diagram where $L$ is a simplicial
complex:
$$\xymatrix{
{L\times\{0\}} \ar[r]^{g} \ar@{^{(}->}[d]_{i_0^m} &  {K\times _BE} \ar[d]^{p'} \\
{L\times I_m} \ar[r]_(.6){G}  & {K} }$$ Considering the
composition of this diagram with the pullback square and using the
fact that $p$ is a simplicial fibration, one can take a simplicial
map $\widetilde{H}\colon L\times I_m\to E$ satisfying $p\circ
\widetilde{H}=f\circ G$ and $\widetilde{H}\circ i_0^m=f'\circ g$.
By the pullback property there is an induced simplicial map
$\widetilde{G}\colon L\times I_m\rightarrow K\times _B E$ making
commutative the following diagram:
$$\xymatrix{
{L\times I_m} \ar@{-->}[dr]^{\widetilde{G}} \ar@/^1pc/[drrr]^{\widetilde{H}}
\ar@/_1pc/[ddr]_{G}   \\
 & {K\times _BE} \ar[rr]^{f'} \ar[d]_{p'} & & {E} \ar[d]^p
  \\
 & {K} \ar[rr]_{f} & & {B}
}$$ By the universal property of the pullback we have that $\widetilde{G}\circ i_0^m=g.$
\end{proof}

\begin{corollary}
Let $K$ and $L$ be simplicial complexes and $K\times L$ be their
categorical product. Then the canonical projections $\pr_1\colon
K\times L\rightarrow K$ and $\pr_2\colon K\times L\rightarrow L$
are simplicial fibrations.
\end{corollary}

\begin{proof}
One has just to take into account  part (ii) of Proposition
\ref{firstprop} because the following square is a pullback:
$$\xymatrix{
{K\times L} \ar[r]^{\pr_2} \ar[d]_{\pr_1} & {L} \ar[d] \\
{K} \ar[r] & {*} }$$
\end{proof}

Another interesting example of simplicial fibration is given by
the product of simplicial fibrations. Recall that, if $f_1\colon
K_1\to L_1$ and $f_2\colon K_2\to L_2$ are simplicial maps, then
one can construct their \emph{product simplicial map}:
$$f_1\times f_2\colon K_1\times K_2\to L_1\times
L_2$$
defined as
$$(f_1\times f_2)(v_1,v_2)\eqdef (f_1(v_1),v_2(v_2)),$$
for any vertex $(v_1,v_2)\in K_1\times K_2.$

\begin{proposition}
Let $p_1\colon E_1\rightarrow B_1$ and $p_2\colon E_2\rightarrow B_2$ be
simplicial fibrations. Then their product $p_1\times p_2$ is also
a simplicial fibration.
\end{proposition}

\begin{proof}
Let $K$ be a  simplicial complex, consider
$$\varphi
=(\varphi_1,\varphi _2)\colon K\times \{0\}\to E_1\times E_2$$ and
$$H=(H_1,H_2)\colon K\times I_m\to B_1\times B_2$$ simplicial maps
such that $(p_1\times p_2)\circ \varphi =H\circ i_0^m$. As
$p_i\circ \varphi _i=H_i\circ i_0^m $ and $p_i$ is a simplicial
fibration, there is a simplicial map $\widetilde{H}_i\colon
K\times I_m\to E_i$ such that $p_i\circ \widetilde{H}_i=H_i$ and
$\widetilde{H}_{i}\circ i_0^m=\varphi _i$ for $i=1,2$. Hence, the
simplicial map
$$\widetilde{H}\colon =(\widetilde{H}_1,\widetilde{H}_2)\colon K\times
I_m\to E_1\times E_2$$ verifies the expected conditions.
\end{proof}



\section{The path complex $\PP K$}\label{S5}

\subsection{Moore paths}

Consider the one-dimensional simplicial complex $Z$, whose
vertices are all the integers $i\in \mathbb{Z}$ and whose
1-simplices are all the consecutive pairs $\{i,i+1\}$, that is,
$\mathbb{Z}$ is a triangulation of the real line.

\begin{definition}[\cite{G}]
Let $K$ be a simplicial complex. A \emph{Moore path} in $K$ is a
simplicial map $\gamma \colon Z\rightarrow K$ which is eventually
constant on the left and eventually constant on the right, i.e.,
there exist integers $i^-,i^+\in \mathbb{Z}$ satisfying the two
following conditions:
\begin{enumerate}
\item[(i)] $\gamma (i)=\gamma (i^-),$ for all $i\leq i^-$,
\item[(ii)]  $\gamma (i)=\gamma (i^+),$ for all $i\geq i^+$.
\end{enumerate}

Obviously, if $i^-=i^+$ we have the constant map.
For a non constant Moore path $\gamma \colon Z\rightarrow K$ we can
consider the integers
\begin{align*}
\gamma^-\eqdef &\max \{i^-\colon \gamma (i)=\gamma
(i^-),\hspace{3pt} \text{for
all\ } i\leq i^-\} \\
\gamma ^+\eqdef &\min \{i^+\colon \gamma (i)=\gamma
(i^+),\hspace{3pt} \text{for all\ }i\geq i^+\}.
\end{align*}
\noindent Observe that $\gamma ^{-}<\gamma ^+$.
\end{definition}

\begin{definition}
The images $\alpha(\gamma)\eqdef \gamma (\gamma ^{-})$ and
$\omega(\gamma)\eqdef \gamma (\gamma ^+)$ are called the
\emph{initial vertex} and \emph{final vertex} of $\gamma $,
respectively. When $\gamma $ is constant we set $\gamma
^{-}=0=\gamma ^+$.
\end{definition}

If $a,b\in \mathbb{Z}$ with $a\leq b$, $[a,b]$ will denote the
full subcomplex of $\mathbb{Z}$ generated by all vertices $i$ with $a\leq
i\leq b$. Considering this notation, any Moore path $\gamma $ in
$K$ may be identified with the restricted simplicial map $\gamma
\colon [{\gamma }^{-},\gamma ^+]\rightarrow K$. The interval
$[\gamma ^-,\gamma ^+]$ will be called the \emph{support} of
$\gamma $.

If $\gamma $ is a Moore path in $K$ with support $[\gamma
^-,\gamma ^+]$, then one can take the \emph{reverse} Moore path
$\overline{\gamma }$ as $$\overline{\gamma}(i)=\gamma (-i),$$
whose support is $[-\gamma ^+,-\gamma ^-]$. Notice that this
reparametrization describes $\gamma$ in the opposite direction.

If $\gamma $ is a Moore path in $K$ with support $[\gamma
^-,\gamma ^+]$ such that $\gamma ^+ - \gamma ^- = m$, then we
define one \emph{normalized} Moore path $|\gamma |\colon I_m\to K$
as
$$|\gamma |(i)=\gamma (i+\gamma^-)\,.$$
The advantage of this reparametrization is that the support of
$|\gamma |$ is $[0,m],$ and therefore it will be more manageable
when dealing with simplicial fibrations.



\begin{definition}
Given $\gamma ,\delta$ Moore paths in $K$ such that
$\omega(\gamma)=\alpha(\delta)$, the \emph{product path} $\gamma
*\delta$ it is defined as
$$(\gamma *\delta)(i)=\begin{cases}
\gamma (i-\delta ^{-}), & i\leq \gamma ^{+}+\delta ^{-},\\
\delta (i-\gamma ^+), & i\geq \gamma ^{+}+\delta ^{-}.
\end{cases}$$
\end{definition}

It is not difficult to see that the support of $\gamma * \delta $
is $[\gamma ^{-}+\delta ^{-},\gamma ^+ +\delta ^+].$ The product
of Moore paths is strictly associative, that is, given Moore paths
$\gamma, \delta, \varepsilon$  such that
$\omega(\gamma)=\alpha(\delta)$ and
$\omega(\delta)=\alpha(\varepsilon)$, then
$$\gamma *(\delta *\varepsilon )=(\gamma *\delta )*\varepsilon.$$

Moreover, if $c_v$ denotes the constant path in a vertex $v\in K$,
then it is immediate to check that $\gamma *c_{w}=\gamma
=c_{v}*\gamma $ where $v=\alpha(\gamma)$ and $w=\omega(\gamma)$.

\subsection{The path complex}

Next we will consider a suitable notion of Moore path complex
associated to a simplicial complex $K$. In order to do so we need
to recall some categorical properties in the category \textbf{SC}
of simplicial complexes and simplicial maps.

Indeed, if $K$ and $L$ are simplicial complexes, we define the
simplicial complex $L^K$, whose vertices are all simplicial maps
$f\colon K\rightarrow L$ and where we consider as simplices the
finite sets $\{f_0,\dots,f_p\}$ of simplicial maps $K\rightarrow
L$ such that
$$\bigcup _{i=0}^pf_i(\sigma )\in L, \quad \text{for any simplex\ } \sigma \in K.$$

It is not difficult to check that this definition induces a structure of
simplicial complex in $L^K$. Moreover, denoting by $\times $ the
categorical product in \textbf{SC}, we have that the {\it
evaluation map} $$\ev\colon L^K\times K\rightarrow K,\quad
(f,v)\mapsto f(v),$$\noindent is simplicial. This fact allows us
to establish a natural bijection
\begin{equation}\label{biject}
\mathbf{SC}(M\times K,L)\equiv
\mathbf{SC}(M,L^K).
\end{equation}

Observe that for a simplicial map
$h\colon L\rightarrow L'$, there is a well defined map
$h^K\colon L^K\rightarrow (L')^K$, which preserves the identities and
the compositions, that is, we have a functor
$(-)^K\colon \mathbf{SC}\rightarrow \mathbf{SC}$. More is true, the
functor $(-)\times K\colon \mathbf{SC}\rightarrow \mathbf{SC}$ is left
adjoint to the functor $(-)^K\colon \mathbf{SC}\rightarrow \mathbf{SC}$.

\begin{definition}
Let $K$ be a simplicial complex. We define the Moore path complex
of $K$, denoted by $\PP K$, as the full subcomplex of $K^Z$
generated by all the Moore paths $\gamma \colon Z \rightarrow K$.
\end{definition}

Then, $\{\gamma _0,\dots,\gamma _p\}\subset \PP K$ defines
a simplex in $\PP K$ if and only if $$\{\gamma _0(i),\dots,\gamma
_p(i),\gamma _0(i+1),\dots,\gamma _p(i+1)\}$$ is a simplex
in $K$, for any integer $i\in \mathbb{Z}$.

An interesting property of $\PP K$ is that, for any bounded interval
$[a,b]\subset \mathbb{Z}$, the complex $K^{[a,b]}$ is, in fact, a full
subcomplex of $\PP K$:
$$K^{[a,b]}\subset \PP K.$$
Moreover, given a simplicial map $f \colon K \to L$, since the
composite $f\circ \gamma$ is a Moore path in $L$ for any Moore
path $\gamma $ in $K$, we obtain the Moore path complex functor
$\PP\colon \mathbf{SC}\rightarrow \mathbf{SC}$. One can check that
this functor preserves binary products and equalizers. Therefore
$\PP$ preserves finite limits and, in particular, pullbacks. In
general $\PP$ does not preserve limits; for instance, $\PP$ does
not preserve infinite products.

\begin{definition}\label{alpha}
The initial and final vertices of any given Moore path $\gamma $
define simplicial maps $\alpha \colon \PP K\to K$ and $\omega
\colon \PP K \to K$.
\end{definition}


\subsection{The path fibration}
\ The aim of this subsection  is to establish the following
important example of simplicial finite-fibration (see Remark
\ref{FINITELIFT}).

\begin{theorem}\label{MAINPATH}
If $K$ is any simplicial complex, then the following simplicial
map
$$\pi=(\alpha ,\omega )\colon \PP K\to K\times K$$ \noindent
is a simplicial finite-fibration where $\alpha$ and $\omega$ are
the maps given in Definition \ref{alpha}.
\end{theorem}

\begin{proof}
Let $L$ be a finite simplicial complex, $m\geq 1$, and a
commutative diagram of simplicial maps
$$\xymatrix{
{L} \ar[rr]^\varphi \ar@{^{(}->}[d]_{i_0^m} & & {\PP K} \ar[d]^{\pi=(\alpha,\omega) } \\
{L\times I_m} \ar[rr]_{(G,H)} & & {K\times K} }$$ \noindent We
recall that, as $L$ is finite, there exists a factorization
of $\varphi$ of the form
$$\xymatrix{
{L} \ar[rr]^\varphi  \ar[dr] & & \PP K \\
 & K^{[a,b]} \ar@{^{(}->}[ur] & }$$
Indeed, if $[a(v),b(v)]$ denotes the support of $\varphi (v)$ for
any vertex $v\in L$, then we may take $a=\min \{a(v)\colon v\in
L\}$ and $b=\max \{b(v)\colon v\in L\}$ due to the fact that $L$ is finite. We therefore obtain a
commutative diagram
$$\xymatrix{
{L} \ar[rr]^\varphi \ar@{^{(}->}[d]_{i_0^m} & & {K^{[a,b]}}
\ar[d]_{(\ev_a,\ev_b)} \ar@{^{(}->}[rr] & & {\PP K} \ar[d]^{\pi = (\alpha,\omega)} \\
{L\times I_m} \ar[rr]_{(G,H)} & & {K\times K} \ar@{=}[rr] & &
{K\times K} }$$ where $\ev_a$ and $\ev_b\colon
K^{[a,b]}\rightarrow K$ denote the evaluation simplicial maps at
$a$ and $b$, respectively. Our aim is to construct a simplicial
map
$$\Omega \colon L\times I_m\to K^{[a-m,b+m]}\subset \PP K$$
satisfying $\pi \circ \Omega =(G,H)$ and $\Omega \circ
i_0^m=\varphi$. Notice that there is natural inclusion $ K^{[a,b]}
\hookrightarrow K^{[a-m,b+m]}$.

\begin{equation}\label{diagram1}
\xymatrix{
{L} \ar@{^{(}->}[d]_{i_0^m} \ar[rr]^\varphi &&{K^{[a,b]}}
\ar@{^{(}->}[rr] && {K^{[a-m,b+m]}} \ar[d]^{(\ev_{a-m},\ev_{b+m})}\ar@{^{(}->}[rr] &&{\PP K} \ar[d]_{\pi } \\
{L\times I_m} \ar@{-->}[urrrr]^{\Omega } \ar[rrrr]_{(G,H)} & &
 & & {K\times K} \ar@{=}[rr] & & {K\times K}
}
\end{equation}

For this task, consider the simplicial maps
$\widehat{G},\widehat{H}\colon L\rightarrow K^{I_m}$ respectively
associated to $G,H\colon L\times I_m \to K$ from the natural
adjunction \eqref{biject}. Then, for any fixed vertex $v\in L$, we
have three Moore paths
\begin{align*}
\varphi(v)\colon& [a,b]\to K, \\
\widehat{G}(v)\colon& I_m\to K, \\
\widehat{H}(v)\colon& I_m\to K,
\end{align*}
satisfying
$\varphi(v)(a)=\widehat{G}(v)(0)$ and $\varphi(v)(b)=\widehat{H}(v)(0)$.

For any fixed $i\in I_m$ we may also consider the $i$-truncated Moore paths, $\widehat{G}(v)_i$ and
$\widehat{H}(v)_i$, whose supports are contained in $[0,i]$ (and
therefore in $[0,m]$), given by:
$$(\widehat{G}(v)_i)(j)=\begin{cases}\widehat{G}(v)(j), & \!\!0\leq j\leq i \\
\widehat{G}(v)(i), & \!\!i\leq j\leq m \end{cases}
\mbox{and}
\;(\widehat{H}(v)_i)(j)=\begin{cases}\widehat{H}(v)(j), & \!\!0\leq j\leq i \\
\widehat{H}(v)(i), & \!\!i\leq j\leq m\end{cases}
$$
Observe that
$\widehat{G}(v)_i$ starts at $\widehat{G}(v)(0)=\varphi(v)(a)$ and ends at
$\widehat{G}(v)(i)=G(v,i)$; moreover, when $i=0$ we obtain the
constant path at $G(v,0)$. Similarly, $\widehat{H}(v)_i$ starts at
$\widehat{H}(v)(0)=\varphi(v)(b)$ and ends at $\widehat{H}(v)(i)=H(v,i);$ we
also have that, when $i=0$, it is the constant path at $H(v,0)$.

Let us denote by $-\widehat{G}(v)_i$ the reverse of the normalized
Moore path of $\widehat{G}(v)_i$, that is,
$$-\widehat{G}(v)_i\eqdef \overline{|\widehat{G}(v)_i|}$$

In this way, $-\widehat{G}(v)_i$, $\varphi(v)$ and
$\widehat{H}(v)_i$ can be multiplied and its multiplication
$(-\widehat{G}(v)_i)*\varphi(v)*\widehat{H}(v)_i$ is a Moore path
starting at $G(v,i)$ and ending at $H(v,i)$, so that for $i=0$ it
equals  $\varphi(v)$. Such multiplication gives rise to the
desired simplicial map $\Omega $ by establishing the identity
$$\Omega
(v,i)(j)\eqdef ((-\widehat{G}(v)_i)*\varphi(v)*\widehat{H}(v)_i)(j).$$ Its explicit
expression is as follows:
$$\Omega (v,i)(j)=
\begin{cases}
G(v,i), & a-m\leq j\leq a-i \\
G(v,a-j), & a-i\leq j\leq a \\
\varphi(v)(j), & a\leq j\leq b \\
H(v,j-b), & b\leq j\leq b+i \\
H(v,i), & b+i\leq j\leq b+m.
\end{cases}$$
At this point we will prove that $\Omega $ is a simplicial map.
Taking into account the exponential law, this is equivalent to
prove that the following map is simplicial:
$$ \widetilde{\Omega}\colon L\times I_m \times [a-m,b+m] \to K, \quad (\sigma, i, j)\mapsto \Omega(\sigma,i)(j). $$
Now, given $\sigma\in L$, $i\in I_m$ and $j\in [a-m,b+m]$, we will prove that
$$\mu\eqdef\widetilde{\Omega}(\sigma\times \{i,i+1\}\times \{j,j+1\})$$
is a simplex in $K$. Taking into account that $\widetilde{\Omega}$
is piecewise defined, we have to consider the following cases:

\begin{itemize}
    \item If $i=0$ and $j=a-i$, then
    $$\mu=\widetilde{\Omega}(\sigma\times \{0,1\}\times \{a,a+1\})=\varphi(\sigma)(\{a,a+1\}).$$

    \item  If $i=0$ and $j=b+i$, then
    $$\mu=\widetilde{\Omega}(\sigma\times \{0,1\}\times \{b,b+1\})=H(\sigma\times \{0,1\}.$$

    \item If $i> 0$ and $j=a-i$, then
    $\mu=G(\sigma\times \{i-1,i\})$.

    \item If $i> 0$ and $j=a-1$, then
    $\mu=G(\sigma\times \{0,1\})$.

     \item If $i> 0$ and $j=b$, then
     $\mu=H(\sigma\times \{0,1\})$.

     \item If $i> 0$ and $j=b+i$, then $\mu=H(\sigma\times \{i,i+1\})$.
\end{itemize}

Since $\varphi$, $G$ and $H$ are simplicial maps, in all cases we
obtain a simplex in $K$ and hence $\widetilde{\Omega}$ is a
simplicial map.

Finally, it is not difficult to check that $\Omega $ satisfies the
commutativity in Diagram (\ref{diagram1}).
\end{proof}

\subsection{$P$-homotopy}

The maps $\alpha$ and $\omega$ of Definition \ref{alpha} allow us
to introduce the following notion of homotopy:

\begin{definition}\label{P-homotop}

Given $f,g\colon K\rightarrow L$ simplicial maps, we will say that $f$
is \emph{P-homotopic} to $g$, denoted by $f\simeq g$, when there
exists a simplicial map
$$H\colon K\to \PP L$$
such that $\alpha \circ H=f$ and $\omega \circ H=g$.

\end{definition}

This relation is certainly reflexive and symmetric but presumably
non transitive (see \cite[p. 123]{G}). More is true, it is compatible with left and right compositions.



\begin{definition}\label{Pequiv}
Let $f\colon K\rightarrow L$ be a simplicial map. Then $f$ is said
to be a \emph{$P$-homotopy equivalence} if there exists a
simplicial map $g\colon L\rightarrow K$ such that $g\circ f\simeq
1_K$ and $f\circ g\simeq 1_L$.
\end{definition}

Taking into account the links between strong homotopy type and contiguity classes established by Barmak and Minian in \cite{LIBROBARMAK} and \cite{B-M}, notice that if $P$-homotopies of the above definition are switched by finite sequences of contiguous maps, then we conclude that $f$ is a strong equivalence. 

An alternative equivalent form for the notion of  $P$-homotopy is
given the following result:

\begin{proposition}
Let $f,g\colon K\rightarrow L$ be simplicial maps. Then $f\simeq
g$ if and only if there exists a sequence of simplicial maps
$\{f_i\colon K\rightarrow L\}_{i\in \mathbb{Z}}$ indexed by the
integer numbers, such that:
\begin{enumerate}
\item[(i)] $f_i\sim _c f_{i+1}$ are contiguous maps;
\item[(ii)] For all vertex $v\in K$ there exist integers $n_v$ and $m_v$ such
that $f_i(v)=f(v),$ for all $i\leq n_v$ and $f_i(v)=g(v)$ for all
$i\geq m_v$.
\end{enumerate}
\end{proposition}

\begin{proof}
If $H\colon K \to \PP L$ is a homotopy between $f$ and $g$ then we
define the simplicial map $f_i\eqdef H(-)(i)\colon K\to L$.
Obviously, as $H$ is simplicial and $\{i,i+1\}$ is a simplex in
$\mathbb{Z}$, we have that $f_i$ and $f_{i+1}$ are contiguous, so condition
(i) holds. Moreover, since for any $v\in K$ we have that $H(v)\in
\PP L$ is a Moore path with support $[n_v,m_v]$, condition (ii) is
also fulfilled as $\alpha \circ H=f$ and $\omega \circ H=g$.

Conversely, given a sequence of simplicial maps
$\{f_i\colon K\to L\}_{i\in \mathbb{Z}}$ satisfying (i) and (ii)
we define $H\colon K \to \PP L$ as
$$H(v)(i)\eqdef f_i(v)$$
It is straightforward to check from (i) that $H$ is simplicial
and from (ii) that $\alpha\circ H=f$ and $\omega \circ H=g$.
\end{proof}


Now we will see the relationship between the homotopy $\simeq $ and the
class of contiguity $\sim $. It is clear that Proposition
\ref{MINIANCONT} can be rewritten as follows: $f$ and $g$ are in
the same class of contiguity if and only if there exist integers
$a\leq b$ and a simplicial map $H\colon K\times [a,b]\rightarrow
L$ such that $H(v,a)=f(v)$ and $H(v,b)=g(v),$ for all vertex $v\in
K$.

\begin{proposition}\label{PSAME}
Let $f,g\colon K\rightarrow L$ be simplicial maps. Then
\begin{enumerate}
\item[(i)] If $f\sim g$,
then $f\simeq g$.

\item[(ii)] If $K$ is finite and $f\simeq g$, then $f\sim g$.
\end{enumerate}
In particular, if
$K$ is finite, then $f$ and $g$ are in the
same contiguity class, $f\sim g$, if and only if they are
$\PP$-homotopic, $f\simeq g$.
\end{proposition}

\begin{proof}
First consider $H\colon K\times [a,b]\rightarrow L$ such that
$H(v,a)=f(v)$ and $H(v,b)=g(v)$. Taking the composition of the
adjoint simplicial map $K\to L^{[a,b]}$, given by the bijection
\eqref{biject}, with the inclusion $L^{[a,b]}\subset \PP L$, we
obtain a simplicial map $H'\colon K\to \PP L$ satisfying
$\alpha\circ H'=f$ and $\omega \circ H'=g$.

Conversely, suppose that $K$ is finite and consider a simplicial
map $G\colon K\to \PP L$ satisfying $\alpha \circ G=f$ and $\omega
\circ G=g$. For any $v\in K$ we have that $G(v)$ is a Moore path in
$L$ with support $[a_v,b_v].$ Taking $a=\min \{a_v\colon v\in K\}$
and $b=\max \{b_v\colon v\in K\}$ we obtain a factorization
$$\xymatrix{
{K} \ar[rr]^G \ar[rd] & & {\PP L} \\
 & {L^{[a,b]}} \ar@{^{(}->}[ur] &
}$$ Considering the adjoint of $K\to L^{[a,b]}$ we obtain a
simplicial map $G'\colon K\times [a,b]\to L$ such that
$G'(v,a)=f(v)$ and $G'(v,b)=g(v),$ for all $v\in K$. Hence, this
means that $f\sim g$.
\end{proof}

We have the following immediate result:

\begin{corollary}
Let $f\colon K\rightarrow L$ be a simplicial map between finite
simplicial complexes. Then $f$ is a $P$-homotopy equivalence if
and only if $f$ is a strong equivalence.
\end{corollary}


\section{Homotopy fiber theorem}
\label{GENFIBER}

Recall that for any natural number $n$, $I_n$ is the full subcomplex of $\mathbb{Z}$ generated by
all vertices $0\leq i\leq n$. If $n^\prime\geq n$, then a
\emph{subdivision map} $t\colon I_{n^\prime}\rightarrow I_n$ is any simplicial map satisfying $t(0)=0$ and $t(n^\prime)=n$.
The proof of the following Lemma can be found in \cite{M2005}.

\begin{lemma}\label{minian}
Given natural numbers $n, m$, there exist $n^\prime, m^\prime$ with
$n^\prime \geq n$ and $m^\prime \geq m$ and a simplicial map $\phi
\colon I_{n^\prime }\times I_{m^\prime }\rightarrow I_n\times I_m$
with the following sketches on the boundaries:

\begin{center}
\includegraphics[scale=0.5]{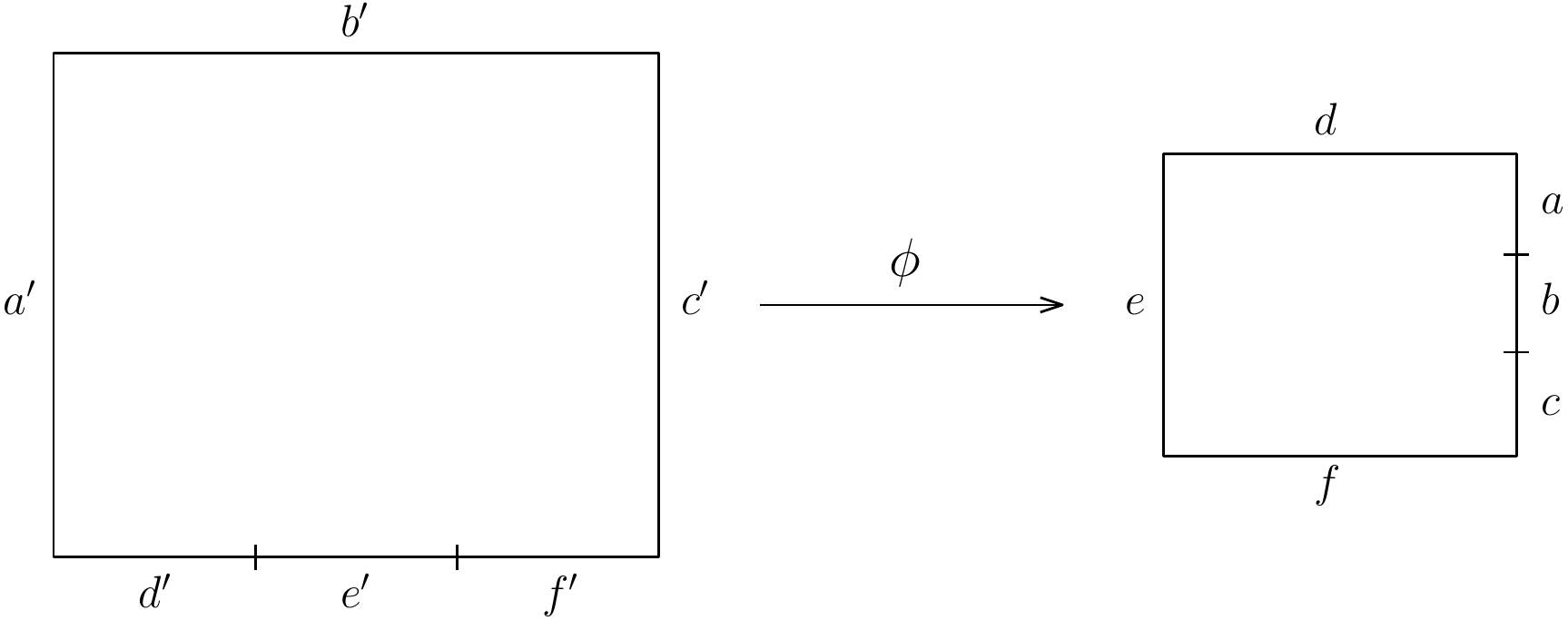}  \\
\end{center}
\bigskip

\noindent where $a^\prime \rightarrow a,$ $b^\prime \rightarrow
b$, etc. are subdivision maps. Moreover, there exist $n^{\prime
\prime}\geq n^\prime$, $m^{\prime\prime}\geq m^\prime $ and a
simplicial map $\phi^\prime \colon I_{n^{\prime\prime}}\times
I_{m^{\prime\prime}}\rightarrow I_{n^\prime }\times I_{m^\prime }$
with the opposite sketches on the boundaries such that the
composition
$$\phi \circ \phi^\prime \colon I_{n^{\prime\prime}}\times
I_{m^{\prime\prime}}\rightarrow I_{n}\times I_{m}$$ has the form
$\phi \circ \phi^\prime =t_1\times t_2,$ where $t_1\colon
I_{n^{\prime\prime}}\rightarrow I_n$ and $t_2\colon
I_{m^{\prime\prime}}\rightarrow I_m$ are subdivision maps.
\end{lemma}

Another technical lemma is the so-called \emph{Simplicial Pasting
Lemma}, whose proof can be found in \cite{Sc-Sw}.

\begin{lemma}\label{pasting}
Let $U,V$ be subcomplexes of a simplicial complex $K$ and let $f
\colon U \to L$, $g\colon V \to L$ be simplicial maps such that
$f(v)=g(v),$ for all vertex $v\in U\cap V$.  Then, the vertex
function $h\colon U\cup V\to L,$ defined as
$$h(v)=\begin{cases}
f(v), & {\rm if\ }v\in U, \\
g(v), & {\rm if\ }v\in V,
\end{cases}$$
is a simplicial map.
\end{lemma}

Given $v,v^\prime  $ vertices in a simplicial complex $K$ and a Moore path $\gamma $
joining $v$ and $v^\prime  $, that is, $\alpha(\gamma)=v$ and
$\omega(\gamma)=v^\prime$, let us consider the Moore path
$\widehat{\gamma}\eqdef |\overline{\gamma }|,$ that is the normalized
of the reverse Moore path of $\gamma $. It is plain to check that
the simplicial map $H\colon I_{2n}\times I_n\to B$ defined as
$$
H(i,j)\eqdef \begin{cases} \gamma (\max \{n-i,j\}), & 0\leq i\leq n, \\
\gamma (\max \{i-n,j\}), & n\leq i\leq 2n,
\end{cases}
$$
\noindent gives a homotopy $H$ between $\widehat{\gamma }*\gamma$ and
the constant Moore path $c_{b^\prime  }$ relative to $\{0,n\}$.
Similarly, one can also check that $\gamma *\widehat{\gamma }\sim c_b$
rel. $\{0,n\}$.

Using this language, a simplicial complex $B$ is \emph{connected}
if and only for any pair of vertices $b,b^\prime  \in B$ there
exists a normalized Moore path $\gamma \colon [0,n]\to B$ such
that $\gamma (0)=b$ and $\gamma (n)=b^\prime  $.

Given a simplicial fibration $f\colon E\to B$ and a vertex $b_0\in
B$, \emph{the simplicial fiber of $p$ over $b_0$}, denoted by
$F_{b_0}$, is the full subcomplex of $E$ generated by all the
vertices $e\in E$ such that $p(e)=b_0$. In other words,
$F_{b_0}=p^{-1}(b_0)$.

The following theorem is one of the main results of this paper,
since it shows that our notion of simplicial fibration has nice
properties like the homotopy invariance of the fiber.
\begin{theorem}\label{PrincipalFibra}
Let $p\colon E\to B$ be a simplicial fibration where $B$ is a
connected simplicial complex. Then any two simplicial fibers of
$p$ have the same strong homotopy type.
\end{theorem}

\begin{proof}
Let $\gamma \colon I_n\to B$ be a Moore path such that $\gamma
(0)=b$ and $\gamma (n)=b^\prime$. Let us first check that there is
a simplicial map $F_{b}\rightarrow F_{b^\prime  }$.  Indeed, if
$i_b$ denotes the inclusion $F_b\subset E$ and $\pr_2\colon
F_b\times I_n\rightarrow I_n$ the natural projection, then take a
lift in the diagram:
$$\xymatrix{
{F_b\times \{0\}} \ar[rr]^{i_b} \ar@{^{(}->}[d] & & {E} \ar[d]^p \\
{F_b\times I_n} \ar[rr]_{\gamma \circ \mathrm{pr}_2}
\ar@{-->}[urr]^{\widetilde{\gamma }} & & {B} }$$ Clearly, we have
that $\widetilde{\gamma }(v,n)\in F_{b^\prime  },$ for all $v\in
F_b$. Therefore there is an induced simplicial map $\gamma
^{\sharp}:F_b\rightarrow F_{b^\prime  }$ given by $\gamma
^{\sharp}(v)\eqdef \widetilde{\gamma }(v,n)$.

Suppose $\delta \colon I_n\rightarrow B$ is another Moore path
satisfying that $\delta (0)=b$, $\delta (n)=b^\prime $ and there
is $H\colon I_n\times I_m\rightarrow B$ such that $H\colon \gamma
\sim \delta $ rel. $\{0,n\}$.  We shall prove that $\gamma^\sharp$
and $\delta^\sharp$ are in the same contiguity class. In
particular we will prove that $\gamma ^{\sharp}$ is independent,
up to contiguity class, of the chosen lift $\widetilde{\gamma }$.

First, let us consider the commutative diagram:
$$\xymatrix{
{F_b\times J_{nm}}  \ar[rr]^(.6)f \ar@{^{(}->}[d] & & {E} \ar[d]^p \\
{F_b\times I_n\times I_m} \ar[rr]_(.6){H\circ \pr} & & {B} }$$
where:
\begin{itemize}
\item $J_{nm}=(I_n\times \{0\})\cup (\{0\}\times
I_m)\cup (I_n\times \{m\})$,
\item $\mathrm{pr}\colon F_b\times I_n\times
I_m\to F_b\times I_n $ is the projection on the two first
complexes,
\item $f$ is the
simplicial map (see Lemma \ref{pasting} above) defined as
\begin{align*}
f(v,i,0)&=\widetilde{\gamma }(v,i), \\
f(v,0,j)&=v, \\
f(v,i,m)&=\widetilde{\delta }(v,i).
\end{align*}
\end{itemize}

Now, take the simplicial maps $\phi \colon I_{n^\prime  }\times
I_{m^\prime  }\rightarrow I_n\times I_m$ and $\phi^\prime  \colon
I_{n^{\prime\prime}}\times I_{m^{\prime\prime}}\rightarrow
I_{n^\prime  }\times I_{m^\prime  }$ given by Lemma \ref{minian}.
As $\phi(I_{n^\prime  }\times \{0\})\subset J_{nm}$, we have a lift
$\widetilde{H}\colon F_b\times I_{n^\prime  }\times I_{m^\prime
}\rightarrow E$, represented by the dotted arrow in the composite
diagram:
$$\xymatrix{
{F_b\times I_{n^\prime  }\times \{0\}} \ar[rr]^{1_{F_b}\times \phi }
\ar@{^{(}->}[d] & & {F_b\times J_{nm}} \ar[rr]^f \ar@{^{(}->}[d] &
&
{E} \ar[d]^p \\
{F_b\times I_{n^\prime  }\times I_{m^\prime  }} \ar[rr]_{1_{F_b}\times \phi }
\ar@{-->}[urrrr]^{\widetilde{H}} & & {F_b\times I_n\times I_m} \ar[rr]_{H\circ \mathrm{pr}} & &
{B} }$$

Taking into account that $\phi^\prime
(J_{n^{\prime\prime}m^{\prime\prime}})\subset I_{n^\prime  }\times
\{0\}$ and that $\phi \circ \phi^\prime  =t_1\times t_2,$ we have
that the simplicial map $$H^\prime  \eqdef \widetilde{H}\circ
(1_{F_b}\times \phi ^\prime  )\colon F_b\times
I_{n^{\prime\prime}}\times I_{m^{\prime\prime}}\rightarrow E$$
\noindent satisfies that $H^\prime  (v,n^{\prime\prime},j)\in
F_{b^\prime },$ for all $v\in F_b$ and $j\in
I_{m^{\prime\prime}}$. Therefore, there is an induced simplicial
map
$$\overline{H}\eqdef H^\prime  (-,n^{\prime\prime},-)^\prime  F_b\times I_{m^{\prime\prime}}\rightarrow F_{b^\prime  }.$$
It is not difficult to see that $\overline{H}(v,0)=\gamma ^{\sharp
}(v)$ and $\overline{H}(v,m^{\prime\prime})= \delta ^{\sharp
}(v);$ that is, $\gamma ^{\sharp }\sim \delta ^{\sharp }$ are in
the same contiguity class.

It is immediate to check that, if $c_b \colon I_n\rightarrow B$
denotes the constant Moore path in $b\in B$, then $c_b^{\sharp }$
is (up to contiguity class) the identity $1_{F_b}\colon
F_b\rightarrow F_b$. Moreover, if $\gamma^\prime \colon
I_n\rightarrow B$ and $\delta ^\prime \colon I_m\rightarrow B$ are
Moore path spaces such that $\gamma (0)=b,$ $\gamma (n)=b^\prime
=\delta (0)$ and $\delta (m)=b^{\prime\prime},$ then we can prove
that
$$(\gamma
*\delta )^{\sharp }\colon F_b\to F_{b^{\prime\prime}}$$ belongs to
the same contiguity class of the composition $\delta
^{\sharp}\circ \gamma ^{\sharp }$. Indeed, the simplicial map
$\widetilde{\gamma *\delta }\colon F_b\times I_{n+m}\ \to E$
defined as
$$(\widetilde{\gamma *\delta })(v,i)=\begin{cases}\widetilde{\gamma }(v,i), & 0\leq i\leq n \\
\widetilde{\delta}(\widetilde{\gamma }(v,n),i-n), & n\leq i\leq
n+m\end{cases}$$ gives a lift:
$$\xymatrix{
{F_b\times \{0\}} \ar[rr]^{i_b} \ar@{^{(}->}[d] & & {E} \ar[d]^p \\
{F_b\times I_{n+m}} \ar[rr]_{(\gamma *\delta )\circ \pr_2} \ar@{-->}[urr]^{\widetilde{\gamma *\delta }}
& & {B} }$$

This proves that $(\gamma *\delta )^{\sharp }$ and $\delta
^{\sharp}\circ \gamma ^{\sharp }$ are in the same contiguity class.

Using the above reasonings and the fact that $\overline{\gamma
}*\gamma \sim c_{b^\prime  }$ rel. $\{0,n\}$ and $\gamma
*\overline{\gamma }\sim c_{b}$ rel.~$\{0,n\}$, we conclude the proof
of the result.
\end{proof}


\section{Collapsible base}\label{S7}

Minian's Lemma (Lemma \ref{minian}) and the Simplicial Pasting Lemma
(Lemma \ref{pasting}) will be crucial for the next results. We
start with this fairly general proposition.

\begin{proposition}\label{main}
Let $p\colon E\rightarrow B$ be a simplicial fibration and let
$F_0,F_1\colon K\times I_n\rightarrow E$ be simplicial maps such
that $p\circ F_0$ and $p\circ F_1$ are in the same contiguity
class with $m$ steps. Let $H\colon p\circ F_0\sim p\circ F_1$
be the map given by Proposition \ref{MINIANCONT} and analogously,
let $G\colon F_0|_{K\times \{0\}}\sim F_1|_{K\times \{0\}}$ be the
map given by Proposition \ref{MINIANCONT}. Consider the following
commutative diagram:
$$\xymatrix{
{K\times \{0\}\times I_m}  \ar[rr]^(.6)G \ar@{^{(}->}[d] & & {E} \ar[d]^p \\
{K\times I_n\times I_m} \ar[rr]_(.6)H & & {B} }$$ \noindent

Then, for suitable $q\geq n$ and $p\geq m,$ there exist
subdivision maps $t_1\colon I_q\rightarrow I_n$ and $t_2\colon I_p\rightarrow
I_m$ and a simplicial map
$$H\sp\prime \colon K\times I_q\times I_p\rightarrow E$$
\noindent such that $H^\prime $ defines a homotopy
$F_0 \circ (1_K\times t_1)\sim F_1 \circ (1_K\times t_1)$, which is an extension of
$G\circ (1_{K\times \{0\}}\times t_2).$
\end{proposition}

\begin{proof}
Consider the commutative diagram
$$\xymatrix{
{K\times J_{nm}}  \ar[rr]^(.6)f \ar@{^{(}->}[d] & & {E} \ar[d]^p \\
{K\times I_n\times I_m} \ar[rr]_(.6)H & & {B} }$$ \noindent where
$J_{nm}$ is the full subcomplex of $I_n\times I_m$ given in the
proof of Theorem \ref{PrincipalFibra} \noindent and $f$ is the
simplicial map (see Lemma \ref{pasting}) defined as
\begin{align*}
f(v,i,0)&=F_0(v,i), \\
f(v,0,j)&=G(v,0,j), \\
f(v,i,m)&=F_1(v,i).
\end{align*}
Now we take the simplicial maps $\phi \colon I_{n\prime }\times
I_{m\prime }\rightarrow I_n\times I_m$ and $\phi\sp\prime \colon
I_{n^{\prime\prime}}\times I_{m^{\prime\prime}}\rightarrow
I_{n^\prime }\times I_{m^\prime }$ given by Lemma \ref{minian}. As
$\phi \circ (I_{n^\prime }\times \{0\})\subset J_{nm}$ we have a
lift $\widetilde{H}\colon K\times I_{n^\prime }\times I_{m^\prime
}\rightarrow E$, represented by the dotted arrow in the composite
diagram:
$$\xymatrix{
{K\times I_{n^\prime }\times \{0\}} \ar[rr]^{1_K\times \phi }
\ar@{^{(}->}[d] & & {K\times J_{nm}} \ar[rr]^f \ar@{^{(}->}[d] & &
{E} \ar[d]^p \\
{K\times I_{n^\prime }\times I_{m^\prime }} \ar[rr]_{1_K\times \phi }
\ar@{-->}[urrrr]^{\widetilde{H}} & & {K\times I_n\times I_m} \ar[rr]_H & & {B} }$$
Renaming $q=n^{\prime\prime}$, $p=m^{\prime\prime}$, and taking into account that $$\phi
^\prime (J_{n^{\prime\prime}m^{\prime\prime}})\subset I_{n^\prime }\times \{0\}$$ \noindent and that $\phi
\circ \phi\sp\prime =t_1\times t_2,$ we have that the simplicial map
$H^\prime \eqdef \widetilde{H} \circ (1_K\times \phi\sp\prime )$ fulfills all the expected
requirements.
\end{proof}

Now, for our next result, we need to establish the following natural
definition:

\begin{definition}
Let $p\colon E\rightarrow B$ be a simplicial fibration and consider
$f,g\colon K\rightarrow E$ simplicial maps. Then it is said that $f$ and
$g$ are \emph{in the same class of fibrewise contiguity}, denoted
by $f\sim _p g,$ if there exists a finite sequence of simplicial
maps
$$f=f_0\sim _c  \dots \sim _c f_n=g$$ \noindent such that
$p \circ f=p \circ f_i$, for all $i=0,\dots,n$.
\end{definition}

In other words, $f$ and $g$ are in the same class of fibrewise
contiguity if there exists a simplicial map $F\colon K\times
I_n\rightarrow X$ satisfying
\begin{enumerate}
\item[(i)] $F(v,0)=f(v)$ and $F(v,n)=g(v)$, for all $v\in K;$

\item[(ii)] $p(F(v,i))=p(f(v))$, for all $v\in K$ and $i\in I_n$
(therefore $p \circ f=p \circ g$).
\end{enumerate}

As a corollary of Proposition \ref{main}, we have:

\begin{corollary}\label{twolifts}
Let $p\colon E\rightarrow B$ be a simplicial fibration and let
$F_0,F_1\colon K\times I_n\rightarrow E$ be two lifts of the same
map
$$\xymatrix{
{K\times \{0\}} \ar[rr]^f \ar@{^{(}->}[d] & & {E} \ar[d]^p \\
{K\times I_n} \ar[rr] \ar@{-->}[urr]^{F_0}_{F_1} & & {B} }$$ Then,
for a suitable $q\geq n,$ there exists a subdivision map
$t\colon I_q\rightarrow I_n$ such that $F_0 \circ (1_K\times t)\sim
_p F_1\circ (1_K\times t).$
\end{corollary}

\begin{proof}
Just apply Proposition \ref{main} to the simplicial maps $G,H$
defined as $G(v,0,j)=f(v,0)$ and $H(v,i,j)=(p \circ f)(v,0),$ for
all $v\in K$ and $j\in I_n$.
\end{proof}

\begin{remark}\label{ene-por-cero}
Observe that the previous corollary also holds true when we
consider $K\times \{n\}$ in the diagram instead of $K\times\{0\}.$
In this case one just have apply the corollary by carefully using
the simplicial isomorphism $I_n\stackrel{\cong }{\rightarrow }I_n$
given by $i\mapsto n-i$. Using this trick one can also check that
a fibration $p\colon E\rightarrow B$ may be also characterized by
the existence of a lift in any diagram of the form
$$\xymatrix{
{K\times \{n\}} \ar[rr] \ar@{^{(}->}[d] & & {E} \ar[d]^p \\
{K\times I_n} \ar[rr] \ar@{-->}[urr] & & {B} }$$\end{remark}

For our next result we will consider the notion of having the same
type of fibrewise contiguity.

\begin{definition}
We say that two simplicial fibrations $p_1\colon E_1\rightarrow
B$, $p_2\colon E_2\rightarrow B$ have the \emph{same type of
fibrewise contiguity} when there exist simplicial maps $f\colon
E_1\rightarrow E_2$ and $g\colon E_2\rightarrow E_1$ such that
$p_2 \circ f=g$, $p_1 \circ g=f$ and $g \circ f\sim _{p_1}1_{E_1}$
and $f \circ g\sim _{p_2}1_{E_2}$.
\end{definition}

The following important result asserts that simplicial maps in the
same class of contiguity induce simplicial fibrations having the
same type of fibrewise contiguity.

\begin{theorem}\label{big-theorem}
Let $p\colon E\rightarrow B$ be a simplicial fibration and let
$f_0,f_1\colon K\rightarrow B$ simplicial maps. Consider, for each
$i=1,2$, the pullback of $f_i$ along $p$:
$$\xymatrix{
{E_i} \ar[d]_{p_i} \ar[r]^{f^\prime _i} & {E} \ar[d]^p \\ {K}
\ar[r]_{f_i} & {B} }$$ If $f_0$ and $f_1$ are in the same class of
contiguity, then the simplicial fibrations $p_0$ and $p_1$ have
the same type of fibrewise contiguity.
\end{theorem}

\begin{proof}
Consider a simplicial map $F\colon K\times I_n\rightarrow B$ such
that $F(v,0)=f_0(v)$ and $F(v,n)=f_1(v),$ for all $v\in K.$ Take
$F^\prime _0$, $F^\prime _1$ lifts of the following diagrams (for
the second diagram see Remark \ref{ene-por-cero}):
$$\xymatrix{
{E_0\times \{0\}}  \ar[rr]^{f^\prime _0} \ar@{^{(}->}[d] & & {E} \ar[d]^p
& &  {E_1\times \{n\}}  \ar[rr]^{f^\prime _1} \ar@{^{(}->}[d] & & {E}
\ar[d]^p   \\
{E_0\times I_n} \ar@{-->}[urr]^{F^\prime _0} \ar[rr]_{F\circ(p_0\times 1)} & &
{B} & & {E_1\times I_n} \ar@{-->}[urr]^{F^\prime _1} \ar[rr]_{F\circ(p_1\times
1)} & & {B}  }$$
By the pullback property there are well defined
simplicial maps $g_0\colon E_0\rightarrow E_1$ and $g_1\colon E_1\rightarrow
E_0$ satisfying the commutativities given in the following
diagrams:
$$\xymatrix{
{E_0} \ar@{-->}[dr]^{g_0} \ar@/^1pc/[drr]^{F^\prime _0(-,n)}
\ar@/_1pc/[ddr]_{p_0} & &    & & {E_1} \ar@{-->}[dr]^{g_1}
\ar@/^1pc/[drr]^{F^\prime _1(-,0)} \ar@/_1pc/[ddr]_{p_1} & &  \\
 & {E_1} \ar[r]^{f^\prime _1} \ar[d]_{p_1} & {E} \ar[d]^p & &
 & {E_0} \ar[r]^{f^\prime _0} \ar[d]_{p_0} & {E} \ar[d]^p \\
 & {K} \ar[r]_{f_1} & {B} & &
 & {K} \ar[r]_{f_0} & {B}
}$$ Let us first check that $g_0 \circ g_1\sim _{p_1}1_{E_1}.$
Indeed, observe that $F^\prime _0 \circ (g_1\times 1)$ and
$F^\prime _1$ are lifts of the same map
$$\xymatrix{
{E_1\times \{0\}} \ar[rr]^{f^\prime _0\circ g_1} \ar@{^{(}->}[d] & & {E} \ar[d]^p \\
{E_1\times I_n} \ar[rr]_{F\circ (p_1\times 1)} \ar@{-->}[urr] & &
{B} }$$ Corollary \ref{twolifts} assures the existence of a
suitable subdivision map $t\colon I_{n^\prime }\rightarrow I_n$
and a simplicial map $G\colon E_1\times I_{n^\prime }\times
I_m\rightarrow E$ satisfying that
$$G\colon F^\prime _0 \circ (g_1\times 1)\circ(1\times t)\sim _p F^\prime _1\circ (1\times t)$$
Again, using the pullback property, there is an induced simplicial
map $\widetilde{G}\colon E_1\times I_m\rightarrow E_1$ satisfying
$$\xymatrix{
{E_1\times I_m} \ar@{-->}[dr]^{\widetilde{G}}
\ar@/^1pc/[drr]^{G(-,n^\prime ,-)}
\ar@/_1pc/[ddr]_{p_1\circ \pr}   \\
 & {E_1} \ar[r]^{f^\prime _1} \ar[d]_{p_1} & {E} \ar[d]^p
  \\
 & {K} \ar[r]_{f_1} & {B}
}$$ A simple inspection proves that $\widetilde{G}\colon g_0 \circ
g_1\sim _{p_1}1_{E_1}.$

Analogously, by applying Corollary \ref{twolifts} to the diagram
$$p\circ f^\prime _1\circ g_0=F\circ (p_0\times 1)|_{E_0\times \{n\}}$$ for the common lifts
$F^\prime _1\circ(g_0\times 1)$ and $F^\prime _0$, and taking into
account Remark \ref{ene-por-cero}, one can find a simplicial map
$H$ satisfying $$H\colon F^\prime _1\circ (g_0\times 1)\circ
(1\times s)\sim _p F^\prime _0\circ (1\times s)$$ for a suitable
subdivision map $s\colon I_{n^{\prime\prime}}\rightarrow I_n.$
Taking $\widetilde{H}$ the simplicial map characterized by the
equalities $f^\prime _0 \circ \widetilde{H}=G(-,0,-)$ and $p_0
\circ \widetilde{H}=p_0\circ \pr$ we obtain $\widetilde{H}\colon
g_1\circ g_0\sim _{p_0}1_{E_0}.$
\end{proof}

\begin{corollary}
Let $p\colon E\rightarrow B$ be a simplicial fibration where $B$
is strongly collapsible. Then $p$ has the same class of fibrewise
contiguity of the trivial fibration $B\times
p^{-1}(b_0)\rightarrow B$, for any vertex $b_0\in B.$
\end{corollary}

\begin{proof}
Just take into account that $1_B\colon B\rightarrow B$ and the
constant path $c_{b_0}\colon B\rightarrow B$ are in the same class
of contiguity and use Theorem \ref{big-theorem}.
\end{proof}

We finish this section with an interesting property whose proof
relies on the proof of Theorem \ref{big-theorem}. Indeed, remember
that, in the statement of such theorem, we have pullbacks
($i=0,1$):
$$\xymatrix{
    {E_i} \ar[d]_{p_i} \ar[r]^{f^\prime _i} & {E} \ar[d]^p \\ {K}
    \ar[r]_{f_i} & {B} }$$ \noindent where $f_0\sim f_1$ (i.e., $f_0$
and $f_1$ are in the same class of contiguity). From the proof one
obtains simplicial maps $g_0:E_0\rightarrow E_1$ and
$g_1:E_1\rightarrow E_0$ (with $p_1\circ g_0=p_0$ and $p_0\circ
g_1=p_1$) and Minian simplicial homotopies
$$F'_0:E_0\times I_n\rightarrow E;\hspace{5pt}F'_1:E_1\times
I_n\rightarrow E$$ \noindent satisfying $F'_0:f'_0\sim f'_1\circ
g_0$ and $F'_1:f'_0\circ g_1\sim f'_1$.  Moreover, $g_0$ and $g_1$
satisfy $g_1\circ g_0\sim _{p_0}1_{E_0}$ and $g_0\circ g_1\sim
_{p_1}1_{E_1}.$ In particular $g_1\circ g_0\sim 1_{E_0}$ and
$g_0\circ g_1\sim 1_{E_1}.$

We easily have the following lemma:

\begin{lemma}
    Consider the pullback of a simplicial fibration $p:E\rightarrow B$
    along a simplicial map $f:B\rightarrow B$ such that $f\sim 1_B:$
    $$\xymatrix{
        {E'} \ar[d]_{p'} \ar[r]^{f'} & {E} \ar[d]^p \\ {B} \ar[r]_{f} &
        {B} }$$ Then $f'\sim g_0$ and $f'\circ g_1\sim 1_E.$ In
    particular, $f'$ is a strong equivalence with $g_1:E\rightarrow
    E'$ as a homotopy inverse.
\end{lemma}

\begin{proof}
    Just observe that the following square is a pullback
    $$\xymatrix{
        {E} \ar[d]_{p} \ar[r]^{1_E} & {E} \ar[d]^p \\ {B} \ar[r]_{1_B} &
        {B} }$$ \noindent and apply the argument above.
\end{proof}

And finally our result. Observe that such result completes
Proposition \ref{pullback}.

\begin{proposition}
    Consider the pullback of a simplicial fibration $p:E\rightarrow B$
    along a simplicial map $f:K\rightarrow B$:
    $$\xymatrix{
        {E'} \ar[d]_{p'} \ar[r]^{f'} & {E} \ar[d]^p \\ {K} \ar[r]_{f} &
        {B} }$$ If $f$ is a strong equivalence, then $f'$ is also a strong
    equivalence.
\end{proposition}

\begin{proof}
Suppose $g:B\rightarrow K$ a simplicial map such that $g\circ
f\sim 1_K$ and $f\circ g\sim 1_B.$ Consider first the following
diagram, where $E''$ is the pullback of $p'$ along $g$:
$$\xymatrix{
{E''} \ar[r]^{g'} \ar[d]_{p''} & {E'} \ar[d]^{p'} \ar[r]^{f'} &
{E} \ar[d]^p \\
{B} \ar[r]_g & {K} \ar[r]_f & {B} }$$ As $f\circ g\sim 1_B$, by
the lemma above (note that the composite of pullbacks is a
pullback) one can find simplicial maps $g_0:E''\rightarrow E$ and
$g_1:E\rightarrow E''$ such that $g_1\circ g_0\sim 1_{E''}$,
$g_0\circ g_1\sim 1_E$, $f'\circ g'\sim g_0$ and $f'\circ g'\circ
g_1\sim 1_E$. Therefore, $g'$ has a left homotopy inverse and $f'$
has a right homotopy inverse:
$$(g_1\circ f')\circ g'\sim 1_{E''}\hspace{8pt}\mbox{and}\hspace{8pt}f'\circ (g'\circ g_1)\sim 1_E$$
Similarly, from the diagram of pullbacks
$$\xymatrix{
{E'''} \ar[r]^{f'''} \ar[d]_{p'''} & {E''} \ar[d]^{p''}
\ar[r]^{g'} &
{E'} \ar[d]^{p'} \\
{K} \ar[r]_f & {B} \ar[r]_g & {K} }$$ \noindent we have that $g'$
has a right homotopy inverse. It straightforwardly follows that
$g'$ is a strong equivalence. Since $f'\circ g'\sim g_0$ we
conclude that $f'$ is a strong equivalence.
\end{proof}


\section{Varadarajan's theorem}\label{S8}
As an application of Theorem \ref{PrincipalFibra} about the strong
homotopy type of the fibers of a fibration, we prove in this
section a simplicial version of a well-known result from
Varadarajan for topological fibrations \cite{VARADARAJAN},
establishing a relationship between  the LS-categories of the
total space, the base and the homotopic fiber. A more general
result for smooth foliations was proved by Colman and Mac\'{\i}as
in  \cite{COLMANMACIAS}.

\begin{definition}\label{SLS}
The simplicial LS-category $\scat(K)$ of the simplicial complex
$K$ is the least integer $n\geq 0$ such that $K$ can be covered by
$n+1$ subcomplexes $K_j$ such that each inclusion $\iota_j\colon
K_j\subset K$ belongs to the contiguity class of some constant map
$c_{v_j}\colon K_j \to K$.
\end{definition}


This notion is the simplicial version of the well known homotopic
invariant $\cat(X)$, the so-called Lusternik-Schnirelmann category
of the topological space $X$ \cite{CLOT}. It has been introduced
by the authors in \cite{F-M-V,F-M-M-V1}, its most important
property being the invariance by strong homotopy equivalences.

Accordingly to Theorem \ref{PrincipalFibra}, all the fibers
$p\inv(v)$, $v\in B$, of a fibration with connected base $p\colon E \to B$ have the
same strong homotopy type, so we call {\em generic fiber} $F$ of
the fibration any simplicial complex into that equivalence class,
and its simplicial category $\scat(F)$ is well defined.

\begin{theorem}\label{Varadarajan}
Let $p\colon E \to B$ be a simplicial fibration with connected
base $B$ and generic fiber $F$.Then
$$\scat (E)+1\leq (\scat (B) +1)(\scat (F)+1).$$
\end{theorem}


\begin{proof}
Let $\scat B=m$, and take a categorical covering $U_0,\dots,U_m$
of $B$.  From Theorem \ref{PrincipalFibra} we know that all the
fibers have the same strong homotopy type. We identify $F$ to the
fiber $p\inv(v)$ over some base point $v\in B$. Let $\scat F=n$,
with $V_0,\dots,V_n$ a categorical covering of $F$.

For each $i\in\{0,\dots,m\}$ let $I_i\colon U_i\subset B$ be the
inclusion. By definition of simplicial category, the map $I_i$ is
in the contiguity class of a constant map, say $c_i\colon U_i \to
B$. Since $B$   is path connected we can assume that $c_i$ is the
constant map  corresponding to the base point $v$. Consider the
map
$$p_i=I_i\circ p\colon p\inv(U_i)\to B.$$
If  $\epsilon_i\colon p\inv(U_i)\subset E$ is the inclusion, then
$p\circ \epsilon_i=p_i$. Now, from $I_i\sim c_v$ it follows that
$p_i\sim c_v$, the latter being the constant map with domain
$p\inv(U_i)$. By the contiguity lifting property, there exists
$$G_i\colon p\inv(U_i)\to E$$ such that $\epsilon_i \sim G_i$ and $p\circ G_i=c_v$.
The latter means that the  map $G_i$ takes its values in $F$. We
denote
$$g_i\colon p\inv(U_i)\to F$$ the map given by $g_i(v)=G_i(v)$.
In this way $G_i=\iota_F\circ g_i$, where $\iota_F\colon F \subset E$ is the inclusion.

For each $i\in \{0,\dots,m\}$, $j\in\{0,\dots,n\}$, we take the subcomplex
$$W_{ij}=g_i\inv(V_j)\subset p\inv(U_i)\subset E.$$
Since $B=U_0\cup\cdots\cup U_m$ implies
$E=p\inv(U_0)\cup\cdots\cup p\inv(U_m)$, and since
$F=V_0\cup\cdots\cup V_n$ implies $p\inv(U_i)=W_{i0}\cup\cdots\cup
W_{i n}$, it follows that $\{W_{ij}\}$ is a   covering of $E$.

It only remains to prove that each $W_{ij}$ is categorical in $E$.
Let $g_{ij}\colon W_{ij}\to V_j$  be the restriction of $g_i$ to
$W_{ij}\subset p\inv(U_i)$. We know that each $V_j\subset F$ is
categorical, so the inclusion $J_j\colon V_j\subset F$ is in the
contiguity class of some constant map $c_j\colon V_j \to F$, whose
image is some vertex $f_j$ of $F$. Then the composition
$$ W_{ij}\stackrel{g_{ij}}{\rightarrow}V_i\stackrel{J_j}{\subset}F,$$
belongs to the contiguity class of the constant map $f_j\colon
W_{ij}\to F$ since
$$J_j\circ g_{ij}\sim c_j\circ g_{ij}=f_j.$$
Let $\epsilon_{ij}\colon W_{ij}\subset E$ be the inclusion, that
is, the restriction of $\epsilon_i\colon p\inv(V_j)\subset E$ to
$W_{ij}$. Since $\epsilon_i\sim G_i$ it follows that
$\epsilon_{ij}\sim G_{ij}$, the latter being the restriction of
$G_i$ to $W_{ij}$. Finally we have
$$\epsilon_{ij}\sim G_{ij}=\iota_F\circ  J_j \circ g_{ij} = \iota_F\circ f_j ,$$
so the inclusion $W_{ij}\subset E$ is in the contiguity class of a constant map.
\end{proof}


\section{Factorization}\label{FACTORIZ}

In this section we will see that any simplicial map may be
considered, in a homotopical sense, as a simplicial fibration.

For our main result in this section we will use the fact that
$(\alpha,\omega )\colon \PP K \to K\times K$ is a simplicial
finite-fibration, for any simplicial complex $K$ (see Theorem
\ref{MAINPATH}).

\begin{theorem}
Let $f\colon K \to L$ be a simplicial map. Then there is a
factorization
$$\xymatrix{
{K} \ar[dr]_j \ar[rr]^f & & {L} \\
& P_f \ar[ur]_p & }$$ \noindent where $j$ is a $P$-homotopy
equivalence and $p$ is a simplicial finite-fibration.
\end{theorem}

\begin{proof}
We shall denote by $P_f=K\times _L \PP L$ the pullback of $f\colon
K\to L$ along $\alpha \colon \PP L \to L$, where $\alpha$ is the
initial vertex map. In this way we have a commutative square:
$$\xymatrix{
{P_f} \ar[d]_{\alpha '} \ar[r]^{f'} & {\PP L}
\ar[d]^{\alpha } \\
{K} \ar[r]_f & {L} }$$

As we have previously observed, the simplicial complex $P_f$ is
defined as the full subcomplex of $K\times \PP L$ whose set of
vertices is given as
$$P_f=\{(v,\gamma )\in K\times
\PP L\colon f(v)=\alpha (\gamma )\}$$ and $\alpha^\prime,f^\prime$
are the restrictions of the obvious projections. We define the
simplicial map $j\colon K\to P_f$ as $j(v)\eqdef (v,c_{f(v)}),$
where $c_{f(v)}\colon \mathbb{Z}\to L$ denotes the constant Moore path at
$f(v)$. On the other hand, the simplicial map $p\colon P_f\to L$
is defined as $p(v,\gamma )\eqdef \omega (\gamma )$, where
$\omega$ is the final vertex map. Obviously, we have that $p\circ
j=f$.

Let us first check that $p$ is a simplicial fibration. Indeed,
consider the following commutative diagram:
$$\xymatrix{
{P_f} \ar[rr]^{f^\prime} \ar[d]_{(\alpha^\prime, p)} & &
{\PP L} \ar[d]^{(\alpha,\omega)} \\
{K\times L} \ar[rr]^{f\times 1_L} \ar[d]_{\pr_1} & & {L\times L} \ar[d]^{\pr_1} \\
{K} \ar[rr]^f & & {L} }$$

It is not difficult to check that the bottom diagram is a
pullback. As, by construction, the composition diagram is also a
pullback, we obtain that the top diagram must be a pullback. But
then, being $(\alpha ,\omega )$ a simplicial finite-fibration, the
simplicial map $(\alpha^\prime,p)\colon P_f\to K\times L$ must be
also a simplicial finite-fibration. Therefore, the composition of
the fibrations $p=\pr_2\circ (\alpha^\prime,p)\colon P_f\to L$ is
a finite-fibration.

Next we check that $j\colon K\to P_f$ is a $P$-homotopy
equivalence. As $\alpha^\prime\circ j=1_K$ it only remains to see
that $j\circ \alpha^\prime\simeq 1_{P_f}$. We directly define the
homotopy $H\colon P_f\to \PP P_f$ as $H(v,\gamma )(i)\eqdef
(v,\gamma _i)$ where $\gamma _i\in \PP L$ is defined as
$$\gamma _i(j)=\begin{cases}\gamma (j), & j\leq i \\ \gamma (i), & j\geq
i,
\end{cases}$$
for any $\gamma \in \PP L$ and $i\in \mathbb{Z}$. Observe that
$\gamma _i$ is the $i$-truncated Moore path of $\gamma$ and hence,
$\alpha (\gamma _i)=\alpha (\gamma )$ and $\omega (\gamma
_i)=\gamma (i)$, so $(v,\gamma _i)$ is a well defined element in
$P_f$. Moreover, if $(v,\gamma )\in P_f$ is fixed, we are going to
check that $H(v,\gamma )\in \PP P_f$; equivalently, we have to
prove that $H(v,\gamma )\colon \mathbb{Z}\to P_f$ is a simplicial map and a
eventually constant Moore path, that is, for any integer $i\in
\mathbb{Z}$, the image
$$H(v,\gamma )(\{i,i+1\})=\{(v,\gamma _i),(v,\gamma
_{i+1})\}=\{v\}\times \{\gamma _i,\gamma _{i+1}\}$$ \noindent is a
simplex in $P_f$, or equivalently, that $\{\gamma _i,\gamma
_{i+1}\}$ is a simplex in $\PP L$. But this is true, as, for any
integer $j\in \mathbb{Z}$, the set of vertices $$\{\gamma
_i(j),\gamma _i(j+1),\gamma _{i+1}(j),\gamma _{i+1}(j+1)\}$$
\noindent is precisely $\gamma (\{j,j+1\})$ or $\gamma
(\{i,i+1\})$, depending on whether $j\leq i$ or $j\geq i$. Thus
$H$ is a simplicial map. Now, taking into account that
$$H(v,\gamma )(i)=\begin{cases}(v,c_{f(v)}), & i\leq \gamma ^- \\
(v,\gamma ), & i\geq \gamma ^+ \end{cases}$$ \noindent we have
that $H(v,\gamma )\in \PP P_f$. Moreover, $\alpha \circ H=j\circ
\alpha '$ and $\omega \circ H=1_{P_f}$. Therefore, to conclude
that $H$ is a P-homotopy equivalence between $j\circ
\alpha^\prime$ and $1_{P_f}$, we just need to justify that it is a
simplicial map. For this final task, considering the exponential
law and the fact that $P_f$ is a full subcomplex of $K\times
L^{\mathbb{Z}}$, it suffices to check that the following map is simplicial:
$$\Lambda \colon K\times L^{\mathbb{Z}}\times \mathbb{Z}\to
K\times L^{\mathbb{Z}},\hspace{10pt}(\sigma,\gamma,i)\mapsto (\sigma,\gamma
_i).$$ Now if consider $\sigma $ a simplex in $K$, $\{\gamma
_0,\dots,\gamma _p\}$ a simplex in $L^{\mathbb{Z}}$ and $\{i,i+1\}$ a
simplex in $\mathbb{Z}$, then we have to prove that the image
$$\Lambda (\sigma
\times \{\gamma _0,...,\gamma _p\}\times \{i,i+1\}),$$
that is,
$$\sigma \times \{(\gamma _0)_i,...,(\gamma _p)_i,(\gamma
_0)_{i+1},\dots,(\gamma _p)_{i+1}\}$$ is a simplex in $K\times
L^{\mathbb{Z}},$ or equivalently
$$\{(\gamma _0)_i,\dots,(\gamma
_p)_i,(\gamma _0)_{i+1},\dots,(\gamma _p)_{i+1}\}$$ is a simplex
in $L^{\mathbb{Z}}$. So take any integer $j\in \mathbb{Z}$ and consider
\begin{align*}
 &\{(\gamma _0)_i(j),\dots,(\gamma
_p)_i(j),(\gamma _0)_{i+1}(j),\dots,(\gamma _p)_{i+1}(j), \\
&\quad (\gamma
_0)_i(j+1),\dots,(\gamma _p)_i(j+1),(\gamma
_0)_{i+1}(j+1),\dots,(\gamma _p)_{i+1}(j+1)\}\end{align*}
Then, it
is not difficult to check that this set of vertices is precisely
$$\{\gamma _0(j),\dots,\gamma _p(j),\gamma
_0(j+1),\dots,\gamma _p(j+1)\}$$ or $$\{\gamma _0(i),\dots,\gamma
_p(i),\gamma _0(i+1),\dots,\gamma _p(i+1)\},$$ depending on
whether $j\leq i$ or $j\geq i$. But whatever the case is, we have
a simplex in $L$ because $\{\gamma _0,\dots,\gamma _p\}$ is a
simplex in $L^{\mathbb{Z}}$. We conclude the proof of the theorem.
\end{proof}

As a consequence of this theorem new finite-fibrations appear.

\begin{example}
For instance, for any simplicial complex $K$ and a based vertex
$v_0\in K$ one can consider the full subcomplex of $\PP K$
$$\PP_0 K=\{\gamma \in \PP K\colon \alpha (\gamma )=v_0\}.$$
Observe that this simplicial complex is nothing else than the
construction $P_f$ given in the theorem above, where $f\colon
\{v_0\}\hookrightarrow K$ is the inclusion map. Then, the
simplicial finite-fibration associated to $f$ is precisely the map
$$p\colon \PP_0 K \to K,\quad\gamma \mapsto \omega (\gamma).$$
\end{example}

\begin{example}\label{DIAGFACTORS}
There is also a special factorization that we are specially
interested in.  Although it does not come from the general
construction of the above theorem, one can check that the
following diagram gives a factorization of the diagonal map
$\Delta \colon K\to K\times K$ through a $P$-homotopy
equivalence followed by a simplicial finite-fibration
$$\xymatrix{
{K} \ar[dr]_c \ar[rr]^{\Delta} & & {K\times K} \\
& \PP K \ar[ur]_{(\alpha ,\omega )} & }$$ We already know that
$(\alpha ,\omega )$ is a simplicial finite-fibration. On the other
hand, $c\colon K\to \PP K$ is defined as the simplicial map
sending each $v\in K$ to the constant Moore path at $v$,
$c(v)\eqdef c_v$. This map is, indeed, a $P$-homotopy equivalence:
the simplicial map $\alpha \colon \PP K \to K$ satisfies $\alpha
\circ c=1_K$. Moreover, the homotopy $H\colon \PP K\to \PP^2 K$
defined as $H(\gamma )(i)=\gamma _i$, where $\gamma _i$ is the
$i$-truncated Moore path of $\gamma$, satisfies $\alpha \circ
H=c\circ \alpha $ and $\omega \circ H=1_{\PP K};$ that is,
$H\colon c\circ \alpha \simeq 1_{\PP K}$.
\end{example}


\section{\!\!\v{S}varc genus}\label{SVARCGENUS}

%

In the classical topological setting, the Lusternik-Schnirelmann
category can be seen as a particular case of the so-called
``\v{S}varc genus'' or {\em sectional category} of a continuous
map.

In this Section we adapt this definition to the simplicial
setting.

\begin{definition}\label{SVGEN}
The {\em simplicial \v{S}varc genus} of a simplicial map
$\varphi\colon K \to L$ is the minimun integer $n\geq 0$ such that
$L$  is the union $L_0\cup\dots\cup L_n$ of  $n+1$ subcomplexes,
and for each $j$ there exists a section $\sigma_j$ of $\varphi$,
that is, a simplicial map $\sigma_j\colon L_j \to K$ such that
$\varphi\circ \sigma_j$ is the inclusion $\iota_j\colon L_j\subset
L$.
\end{definition}

We denote this genus by $\Sg(\varphi)$. A slight modification in the above definition is
to change the equality by ``being in the same contiguity class''.

\begin{definition}\label{HSG}
The {\em homotopy simplicial \v{S}varc genus} of a simplicial map
$\varphi\colon K \to L$, denoted by $\hSg(\varphi)$, is the
minimum $n\geq 0$ such that $L=L_0\cup\dots\cup L_n$ , and for
each $j\in\{0,\dots,n\}$ there exists an ``up to contiguity
class'' simplicial section $\sigma_j$ of $\varphi$, that is, a
simplicial map $\sigma_j\colon L_j\to K$ such that $\varphi\circ
\sigma_j\sim \iota_j$.
\end{definition}

\begin{remark}\label{SVARCCONP}
Note that when the complex $L$ is finite, the condition
$\varphi\circ \sigma_j \sim \iota_j$ (same contiguity class) can
be changed to $\varphi\circ \sigma_j \simeq\iota_j$ (P-homotopy),
by Proposition \ref{PSAME}.
\end{remark}

Obviously, $\hSg(\varphi)\leq \Sg(\varphi)$. The equality holds
for some particular classes of maps.

\begin{theorem}\label{homotgenus}
Let $p\colon E \to B$ be a simplicial fibration. Then
$\hSg(p)=\Sg(p)$.
\end{theorem}

\begin{proof}
We only have to prove that $\Sg(\varphi)\leq \hSg(\varphi)$, so
let $\hSg(\varphi)=n$ and $L_0,\dots, L_n$ be a covering of $B$ by
subcomplexes such that there exist simplicial maps $\sigma_j\colon
L_j \to E$ with $p\circ \sigma_j \sim \iota_j$, as in Definition
\ref{HSG}. By Theorem \ref{equiv}, for each $j$ we have simplicial
maps
$$H_j\colon L_j \times [0,m_j]\to B$$ with $H_j(-,0)=p\circ \sigma_j$
and $H_j(-,m_j)=\iota_j$ the inclusion $L_j\subset B$. Take a lift
$\tilde H_j\colon L_j\times [0,m_j]\to E $ in the following
diagram,
\begin{center}
$\xymatrix{
\ \ L_j\times \{0\}\ \ar@{^{(}->}[d] \ar[r]^{\ \ \ \ \sigma_j\ } & E\ar[d]^p\\
L_j\times [0,m_j]\ar[r]^{\ \ \ H_j}\ar@{-->}[ru]^{\tilde H_j}&B\\
}$
\end{center}
in such a way that $p\circ \tilde H=H$ and $\tilde
H(v,0)=\sigma_j(v)$. Then the map $\xi_j\colon L_j \to E$ given by
$\xi_j(v)=\tilde H(v,m_j)$ is simplicial and verifies
$$p\circ \xi_j(v)=p\circ \tilde H(v,m_j)=H(v,m_j)=\iota_j(v)=v,$$
so it is a true section of $p$. Then $\Sg(\varphi)\leq n$.
\end{proof}

\begin{remark}\label{FIBFIN} Note that if $p\colon E \to B$ is
a simplicial finite-fibration then Theorem \ref{homotgenus} above
is still true when the base $B$ is finite.
\end{remark}


%

%

\subsection{Simplicial L-S category}

Taking into account the notion of {\em
simplicial LS-category} (Definition~\ref{SLS}), what we want
to see is that it equals (as in the classical case) the \v{S}varc
genus of a certain fibration.


Assume that the complex $K$ is connected. Then every two
constant maps $c_{v_i}$, $c_{v_j}$, belong to the same contiguity
class, as can be easily seen by considering a path $\gamma \colon
[0,m] \to K$ connecting $v_i$ and $v_j$ and the sequence of
contiguous maps given by the truncated paths $\gamma_k \colon
[0,k] \to K$. So we can choose a base point $v_0$ and assume that
all the constant maps in Definition \ref{SLS} equal $c_{v_0}$.

Let $P_0K$ be the full subcomplex of the path complex $PK$
consisting on all Moore paths whose initial point is the base
point,
$$P_0K=\{\gamma\in PK\colon \alpha(\gamma)=v_0\}.$$

The map $\omega\colon P_0K \to K$, sending each path to its final
point, is then the pullback of the simplicial finite-fibration
$(\alpha,\omega)\colon PK \to K\times K$ by the map $f$, where
$f(v)=(v_0,v)$:

\begin{center}
$\xymatrix{
\ P_0K\ \ar[d]_{\omega}\ar@{^{(}->}[r]&PK\ar[d]^{(\alpha,\omega)}\\
K\ar[r]^{f}&K\times K\\
}$
\end{center}

\begin{theorem}
Let $K$ be a connected finite complex. Then the simplicial
LS-category $\scat(K)$ equals the \v{S}varc genus of the
simplicial finite-fibration $\omega\colon P_0K \to K$.
\end{theorem}

\begin{proof}
We have to prove that $\scat(K)=\Sg(\omega)$, but we know that the
latter equals $\hSg(\omega)$.

So, first, assume that $\scat(K)=n$ and let $K=K_0\cup\dots\cup
K_n$ such that each inclusion $\iota_j\colon K_j\subset K$ belongs
to the contiguity class of the constant map $v_0$,  that is, there
is a Minian simplicial homotopy $H_j\colon K_j\times [0,m_j] \to K$ with
$H_j(v,0)=v_0$ and $H_j(v,m_j)=v$.

Take a lift $\tilde H_j$ of $H_j$ in the following diagram:

\begin{center}
$\xymatrix{
\ \ K_j\times \{0\}\ \ar@{^{(}->}[d] \ar[r]^{\ \ \ \ c_j\ }& P_0K \ar[d]^\omega \\
K_j\times [0,m_j]\ar[r]^{\ \ \ H_j}\ar@{-->}[ru]^{\tilde H_j}&K\\
}$
\end{center}
where $c_j(v)$ is the constant Moore path at $v_0$, for  all  $v$. Define
$\sigma_j\colon K_j\to P_0K$ as $\sigma_j(v)=\tilde H_j(v,m_j)$. Then
$\omega \circ \sigma_j(v)=H_j(v,m_j)=v$, so $\sigma _j$ is a
section of the fibration. The map $\sigma _j$ is simplicial
since 

$$v\in K_j \mapsto (v,m_j)\in K_j\times [0,m_j]$$

\noindent is a simplicial map. We have then proved that $\Sg(\omega)\leq n$.

Conversely, if $s\colon L \to P_0K$ is a section of $\omega$, each
path $s(v)$ has initial vertex $v_0$ and final vertex $v$. Denote by
$[v^-,v^+]$ the support of $s(v)$. Then, since $K$ is finite,
there is an interval $[m,n]$ containing all the supports, so we
can define $H\colon L\times [m,n] \to K$ as $H(v,i)= s(v)(i)$.
Then $H(v,m)=s(v)(m)=s(v)(v^-)=\alpha(s(v))=v_0$, while
$H(v,m)=s(v)(n)=s(v)(v^+)=\omega(s(v))=v$, showing that the
inclusion $L\subset K$ belongs to the contiguity class of the
constant map $v_0$.

We must check that $H$ is simplicial: if $\{v_0,\dots,v_p\}$ is a
simplex of $L$ and $\{i,i+1\}$ is a simplex in $[m,n]$, the image
by $H$ of $\sigma\times \{i,i+1\}$ is the set
$$\{s(v_0)(i),\dots,s(v_p)(i),s(v_0)(i+1),\dots,s(v_p)(i+1)\},$$
which is a simplex of $K$ because $s(\sigma)$ is a simplex of $P_0K$.

So each covering of $K$ by subcomplexes $L$ verifying the
Definition \ref{SVGEN} gives the same covering verifying
Definition \ref{SLS}. Then $\scat(K)\leq \Sg(\omega)$.
\end{proof}

\subsection{Discrete topological complexity}
In \cite{F-M-M-V2}, a subset of the authors introduced a notion of
{\em discrete topological complexity}  which is a version of
Farber's topological complexity \cite{FARBER}, adapted to the
simplicial setting.

By using well known equivalences in the topological setting, the
simplicial definition avoids the use of any path complex.

\begin{definition}\cite{F-M-M-V2}
The {\em discrete  topological complexity} $\TC(K)$ of the
simplicial complex $K$ is the least integer $n\geq 0$ such that
$K\times K$ can be covered by $n+1$ ``Farber subcomplexes''
$\Omega_j$, where the latter means that there exist simplicial
maps $\sigma_j\colon \Omega_j \to K$ such that $\Delta\circ
\sigma_j$ is in the contiguity class of the inclusion
$\iota_j\colon \Omega_j\subset K\times K$. Here $\Delta :K\to K\times K$ denotes the diagonal map $v\mapsto (v,v)$.
\end{definition}

In other words:

\begin{proposition}\label{FIBFIN2}
The discrete topological complexity of the abstract simplicial
complex $K$ is the homotopic \v{S}varc genus of the diagonal map
$\Delta \colon K \to K \times K$, i.e., $\TC(K)=\hSg(\Delta)$.
\end{proposition}

Our main result in this ection is the following one.
\begin{theorem} Let $K$ be a finite complex. The discrete topological complexity of $K$
equals the \v{S}varc genus of the finite-fibration $(\alpha,\omega)\colon PK \to K\times K$.
\end{theorem}
\begin{proof}
As was stated in Example \ref{DIAGFACTORS}, there is a
$P$-equivalence $c$ between the complexes $K$ and $PK$, in such a
way that the diagonal factors through the finite-fibration
$(\alpha,\omega)\colon PK \to K \times K$. Since $K \times K$ is
finite, Remark \ref{SVARCCONP} applies and the homotopic \v{S}varc
genus of $\Delta$ and $(\alpha,\omega)$ can be computed by means
of $P$-homotopies. Then it is clear that
$$\TC(K)=\hSg(\Delta)=\hSg((\alpha,\omega))= \Sg((\alpha,\omega)),$$
where the latter equality follows from Proposition \ref{FIBFIN2}.
\end{proof}


\end{document}